\documentclass[11pt]{article}
\usepackage{geometry}                
\usepackage{float}
\usepackage{graphicx}
\usepackage{amsmath,amssymb,amsthm}

\usepackage{psfrag}
\usepackage{float,subfigure}

\newtheorem{definition}{Definition}
\newtheorem{lemma}{Lemma}

\newtheorem{proposition}{Proposition}
\newtheorem{theorem}{Theorem}

\newtheorem{property}{Property}

\newtheorem{conjecture}{Conjecture}

\newtheorem{remark}{Remark}

\def\cT{{\cal T}}

\def\conj{\hbox{\rm conj}}

\def\R{\mathcal{R}}
\def\Z{\mathcal{Z}}

\newcommand{\abs}[1]{\ensuremath{\vert{#1}\vert}}
\newcommand{\re}{\mathop{\rm Re}\nolimits}
\newfont{\bb}{msbm10}

\def\R{\hbox{\bb R}}

\def\N{\hbox{\bb N}}
\def\C{{\hbox{\bb C}}}

\title{The Construction of Doubly Periodic Minimal Surfaces via Balance Equations}
\author{Peter Connor and Matthias Weber}
\date{November 24, 2009}                                           

\begin{document}
\maketitle

\noindent {\sc Abstract. } {\footnotesize}
Using Traizet's regeneration method, we prove the existence of many new 3-dimensional families of
embedded, doubly periodic minimal surfaces. All these families have a foliation of $\R^3$ by vertical planes as a limit. In the quotient, these limits can be realized conformally as noded Riemann surfaces, whose components are copies of $\C^*$ with finitely many nodes. We derive the balance equations for the location of the nodes and exhibit solutions that allow for surfaces of arbitrarily large genus and number of ends in the quotient. 



\section{Introduction}

A minimal surface $M$ is called {\em doubly periodic} if it is invariant under two linearly independent orientation-preserving translations in euclidean space, which we can assume to be horizontal. The first  such example was discovered by Scherk \cite{sche1}. 

We denote the 2-dimensional lattice generated by the maximal group of such translations by $\Lambda$.
If the quotient $M/\Lambda$ is complete, properly embedded, and of finite topology,  Meeks and Rosenberg \cite{mr3} have shown that  the quotient has a finite number  of annular top and bottom
ends which are asymptotic to flat annuli.

There are two cases to consider: either the top and bottom ends are parallel, or not. By results of Hauswirth and Traizet \cite{hatr1}, a {\em non-degenerate} such surface is a smooth point of a moduli space of dimensions 1 in the non-parallel and 3 in the parallel case.

Moreover, Meeks and Rosenberg \cite{mr3} have shown that in the parallel case, the number of top and bottom ends is equal to the same even number.

Lazard-Holly and Meeks \cite{lm1} have shown that the doubly 
periodic Scherk surfaces are the only embedded doubly periodic surfaces of genus 0.
In particular, the case of parallel ends doesn't occur for this genus.

For genus 1, there is an example of Karcher with orthogonal ends as well as a 3-dimensional family of such surfaces with parallel ends by Karcher \cite{ka4} and Meeks-Rosenberg \cite{mr4}. Moreover, P\'{e}rez, Rodriguez and Traizet  \cite{prt1} have shown that any doubly periodic minimal surface of genus one with parallel ends belongs to this family.

\begin{figure}[H]
	\centerline{ 
		\includegraphics[width=2.2in]{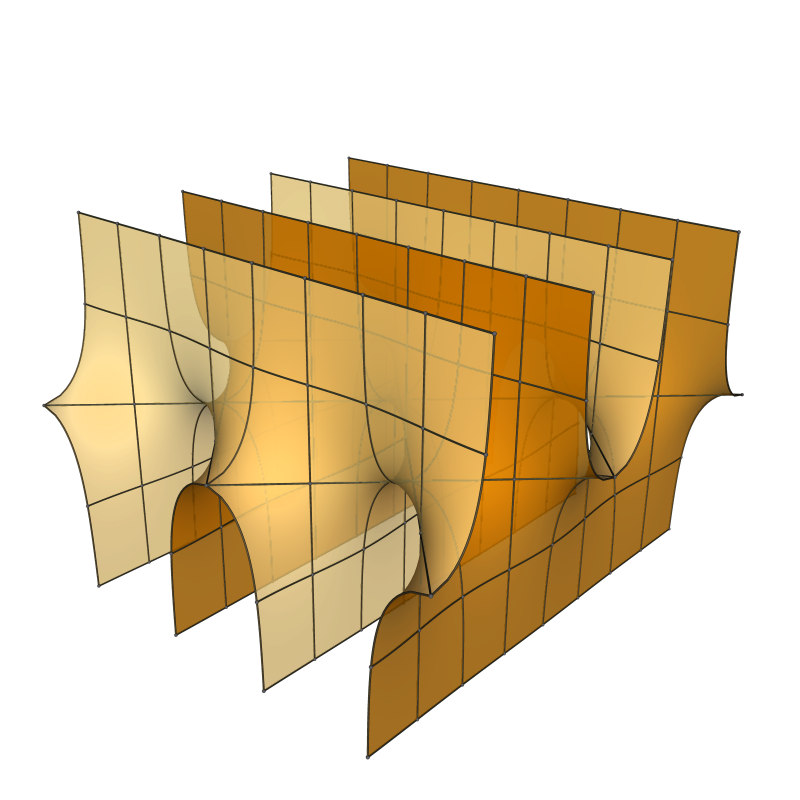}
		\hspace{.5in}
		\includegraphics[width=2.2in]{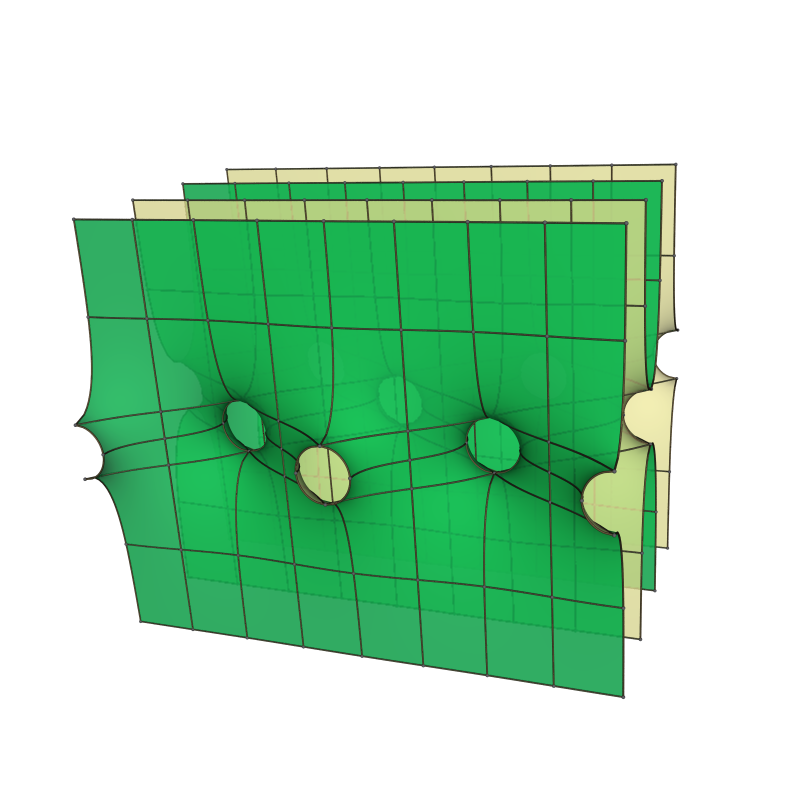}
	}
	\caption{Scherk's surface and a Karcher-Meeks-Rosenberg surface
}
	\label{figure:scherk}
\end{figure}

 Douglas \cite{dou1} and indepently Baginsky and Batista \cite{brb1} have shown that the Karcher example can be deformed to a 1-parameter family by changing the angle between the ends. The family limits in the translation invariant helicoid with handles \cite{howeka4, whw1}
 
For higher genus, only a few examples and families have been known so far:

In the non-parallel case, Weber and Wolf \cite{ww3} have constructed examples of arbitrary genus, generalizing Karcher's example of genus 1. 

Wei found a 1-parameter family of examples of genus 2 with parallel ends \cite{wei2}. This family has been generalized considerably by Rossman, Thayer and Wohlgemuth \cite{rtw1} to allow for more ends.  Rossman, Thayer, and Wohlgemuth did also construct an example with genus 3.

\begin{figure}[H]
	\centerline{ 
		\includegraphics[width=2in]{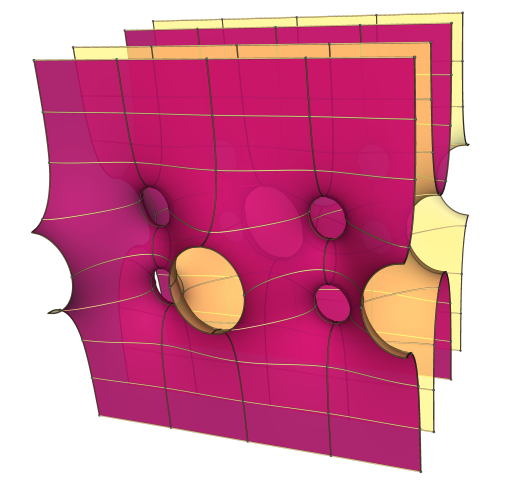}
		\hspace{.5in}
		\includegraphics[width=2in]{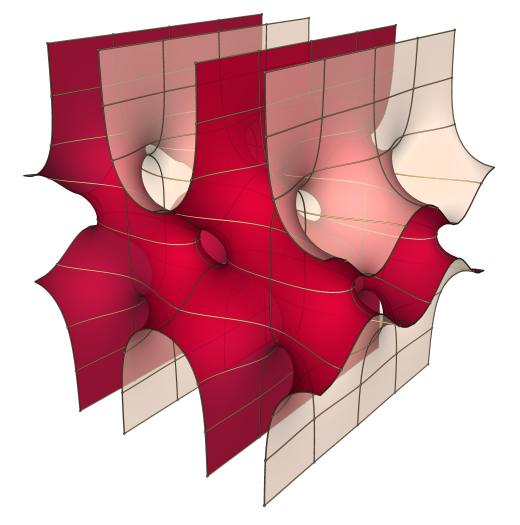}
	}
	\caption{Genus two Wei surface and genus two RTW surface
}
	\label{figure:wei}
\end{figure} 

Our goal is to prove 
\begin{theorem}
\label{thm:main}
For any genus $g\ge1$ and any even number $N\ge 2$, there are
3-dimensional families of complete, embedded, doubly periodic minimal surfaces  in euclidean space
of  genus $g$ and $N$ top and $N$ bottom ends in the quotient.  
\end{theorem}
 
Thus all topological types permitted by the results of Meeks and Rosenberg actually occur.

Figure \ref{figure:ex4} shows two translational copies in each direction of an example of genus 7.

\begin{figure}[H]
	\centerline{	\includegraphics[width=2in]{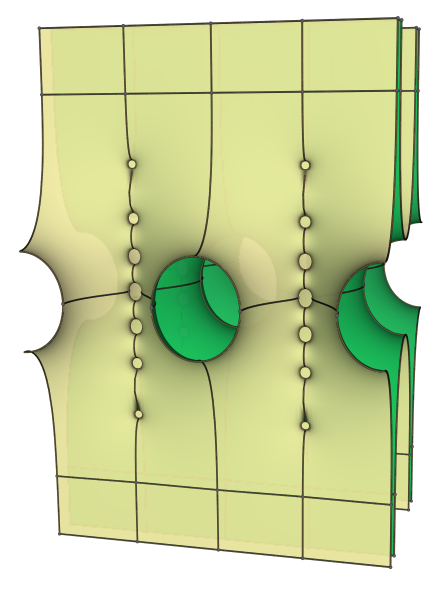}\hspace{1cm}\includegraphics[width=1.8in]{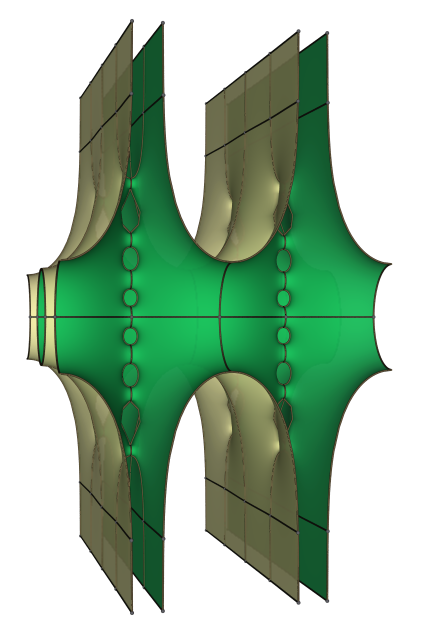}}	
	\caption{Two views of a genus 7 surface}
	\label{figure:ex4}
\end{figure}

The methods used in this paper are an adaptation of Traizet's techniques developed in \cite{tr2}. There, Traizet constructs singly periodic minimal surfaces akin to Rieman's examples which limit in a foliation of euclidean space by horizontal planes. Near the limit, the surfaces look like a collection of parallel planes joined by catenoidal necks. In the limit, these necks develop into nodes so that the quotient surface becomes a noded Riemann surface. The components of the smooth part are punctured spheres, where the punctures have to satisfy Traizet's balance equations. Vice versa, given a finite collection of punctured spheres where the punctures satisfy the balance equations and are non-degenerate in a suitable sense, Traizet constructs a moduli space of Riemann surfaces which forms an open neighborhood of the noded surface. On these Rieman surfaces, he constructs Weierstrass data and solves the period problem using the implicit function theorem.

We will closely follow Traizet's paper, indicating all differences.

The paper is organized as follows: In section 2, we state the results. In section 3, we give examples.  The main theorem is proven in sections 4 through 8.  We prove the embeddedness of our surfaces and show they satisfy certain properties in section 8.


\section{Results}

In this section, we will state precise formulations of our main theorems and introduce the relevant notation.

 \subsection{Description of the surfaces and its properties}
Our goal is to construct three-dimensional families  of embedded doubly periodic minimal surfaces $M$  of arbitrary genus and with an even number $N$ pairs of  annular ends in the quotient.  The surfaces will depend on a small real parameter $t$ (produced by the implicit function theorem) and a complex parameter $T$ explained below.

In contrast to the introduction, we will choose the ends to be horizontal: This allows us to follow the notation and set-up of \cite{tr2} more closely.

Denote the maximal group of orientation preserving translations of $M$ by $\Gamma$. This group will contain a cyclic subgroup of horizontal translations. Denote one of its generators by $\cT$.

By rotating and scaling the surface, 
we can assume that $\cT=(0,2\pi,0)$.  We will identify the horizontal $(x_1,x_2)$-plane with the complex plane $\C$ using $z=x_1+i x_2$. Note that the horizontal planar ends become flat annular ends in the quotient.  Label a non-horizontal generator of $\Lambda$ by $\cT_t$. For $t\to 0$, $\cT_t$ will converge to a horizontal vector $\bar T$, where $T$ is an arbitrary complex parameter. The conjugation is due to orientation issues that will become clear later on. 

Also, order the ends by height and label them $0_k$ and $\infty_k$, with $k\in\Z$.  Most of our work takes place on the quotient surfaces.  
There, the ends will be labeled $0_k$ and $\infty_k$ as well, with $k=1,\ldots,N$ for some even integer $N$.

Our surfaces will have two additional properties.

\begin{property}
The quotient surface $\tilde{M}_t=M_t/\Lambda$ is a union of the following types of domains:  for each pair of ends $E_k=\{0_k,\infty_k\}$, $k=1,\ldots,N$, there is an unbounded domain $E_{k,t}\subset \tilde{M}_t$ containing the ends $0_k$ and $\infty_k$ that is a graph over a domain in $\C^*=\C\setminus\{0\}$ with $n_k+n_{k-1}$ topological disks removed.

$\tilde{M}_t-(  E_k\cup E_{k+1})$ consists of $n_k$ bounded annular components $C_{k,i,t}$ on which the Gauss map is one-to-one, called \emph{catenoidal necks}.  
\label{property1}
\end{property}
\begin{figure}[H]
	\centerline{ 
		\includegraphics[width=2.5in]{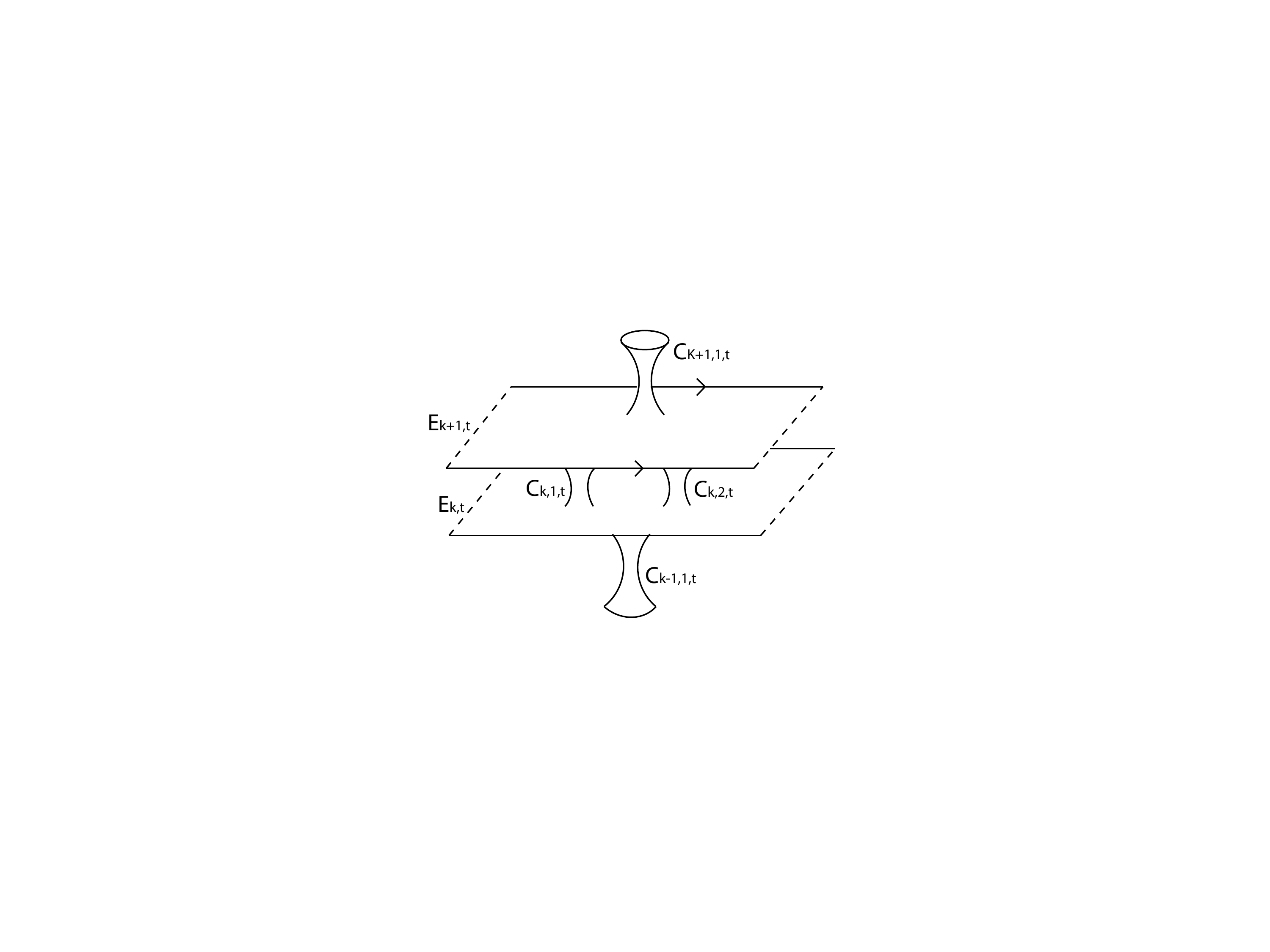}
	}
	\caption{Annular regions and catenoid-shaped necks.
}
	\label{figure:domains}
\end{figure} 
\begin{property}
There is a non-horizontal period $\cT_t$ such that as $t\rightarrow 0$:
\begin{enumerate}
\item
The nonhorizontal period $\cT_t$ converges to a (possibly $0$) horizontal vector $\bar T$.
\item
The surfaces limit in a foliation of $\R^3$ by parallel planes.
\item
The necksize of each annular component $C_{k,i,t}$ shrinks to $0$, and the center of the neck $C_{k,i,t}$ converges to a point $p_{k,i}$.
\item
The underlying Riemann surfaces limit in a noded Riemann surface consisting of $N$ copies of $\C^*=\C\setminus \{0\}$, with nodes at the points $p_{k,i}$.
\end{enumerate}
\label{property2}
\end{property}

Note that when we draw a model of $\tilde{M}_t$, the $E_{k,t}$ components should have the shape of an infinite annulus.  As this is impossible to draw, we model the $E_{k,t}$ components with infinite flat cylinders.

After rotating the KMR and Wei's surfaces so that the ends are horizontal, the behavior of both families near one of their limit fits the description given above.

\subsection{Forces and Balance Equations}

The location of the nodes introduced above is not arbitrary but governed by a system of algebraic equations.

  Consider $N$ copies of $\C^*$, labeled $\C_k^*$ for $k=1,\ldots,N$.    On each $\C_k^*$, place $n_k$ points $p_{k,1},\ldots,p_{k,n_k}$.  Extend this  definition of $p_{k,i}$ for any integer $k$ by making it  periodic with respect to a horizontal vector $T$ in the sense that $p_{k+N,i}=p_{k,i}e^T$ for $k=1,\ldots,N$ and $i=1,\ldots,n_k$, with $n_{k+N}=n_k$.  The difference between our $p_{k,i}$ terms and the ones in \cite{tr2} is that the periodic condition in \cite{tr2} is given by $p_{k+N,i}=p_{k,i}+T$.  The reason for this is that the quotient map for us is given by $\text{exp}:\C\mapsto\C^*$.  Thus, when we look at pictures of our surfaces, the nodes are really located at $\log{p_{k,i}}$ and are subject to the period vector $\log{e^T}=T$.  
  
  This set of points must satisfy a balancing condition given in terms of the following force equations.\\

\begin{definition}
The force exerted on $p_{k,i}$ by the other points in $\{p_{k,i}\}$ is defined by
{\footnotesize \[
F_{k,i}:=\sum_{j \neq i}\frac{p_{k,i}+p_{k,j}}{n_k^2(p_{k,i}-p_{k,j})}+(-1)^k\left(\sum_{j=1}^{n_{k+1}}\frac{p_{k+1,j}^{(-1)^k}}{n_kn_{k+1}\left(p_{k+1,j}^{(-1)^k}-p_{k,i}^{(-1)^k}\right)}-\sum_{j=1}^{n_{k-1}}\frac{p_{k,i}^{(-1)^k}}{n_kn_{k-1}\left(p_{k,i}^{(-1)^k}-p_{k-1,j}^{(-1)^k}\right)}\right).
\]}

\end{definition}

\begin{definition}
The configuration $\{p_{k,i}\}$ is called a \textit{balanced configuration} if $F_{k,i}=0$ for $k=1,\ldots,N$ and $i=1,\ldots,n_k$.
\end{definition}

Note that while the force equations don't seem to contain the parameter $T$, it enters the picture implicitly as the $p_{k,i}$ are assumed to form a $T$-periodic set.

\begin{definition}
Let $m=\sum_{i=1}^Nn_k$ and $F$ and $p$ be the vectors in $\C^m=\R^{2m}$ whose components are made up of the $F_{k,i}$ and $p_{k,i}$ respectively.  The balanced configuration $\{p_{k,i}\}$ is said to be {\it non-degenerate} if the differential of the map $p\mapsto F$ has rank $2(m-1)$.
\end{definition}

The differential of the map $p\mapsto F$ can't have full rank $2m$ because $$\sum_{k=1}^N\sum_{i=1}^{n_k}F_{k,i}=0.$$  This holds whether or not the configuration $\{p_{k,i}\}$ is balanced.

Observe also that whenever we have a solution $p$ for the balance equations, $\lambda p$ will also be a solution for any $\lambda\in\C^*$.

Now, we can state our main result.
\begin{theorem}
If $\{p_{k,i}\}$ is a non-degenerate balanced configuration then there exists a corresponding three-dimensional family of embedded doubly periodic minimal surfaces with genus
$$g=1+\sum_{k=1}^N(n_k-1),$$
$2N$ horizontal ends and properties \ref{property1} and \ref{property2}.
\label{main theorem}
\end{theorem}

Our Main Theorem \ref{thm:main} will follow from this theorem and the non-degeneracy of the balance configurations of Proposition \ref{prop:crucial}.

\section{Examples}
In this section, we will discuss examples of non-degenerate balanced configurations.

\subsection{Adding handles to Wei's genus two examples}

In all known instances of Traizet's regeneration technique, the simplest non-trivial configurations are given as the roots of
special polynomials that satisfy a hypergeometric differential equation. So far, there is no explanation of this phenomenon, neither a general understanding of the more complicated solutions of the balance equations. In the case at hand, we have the following:

\begin{proposition}
Let $n \in \N$ and $a_1,a_2, \cdots, a_n$ be the roots of the polynomial 
\[
p_n(z)=\sum_{k=0}^n{n \choose k}^2z^k. 
\]
The following configuration is balanced and non-degenerate: $N=2$, $n_1=1$, $n_2=n$, $p_{1,1}=1$, $p_{2,i}=a_i$ for $i=1,\cdots,n$, and $T=0$.
\label{handles}
\end{proposition}
\begin{figure}[H]
	\centerline{ 
		\includegraphics[width=3in]{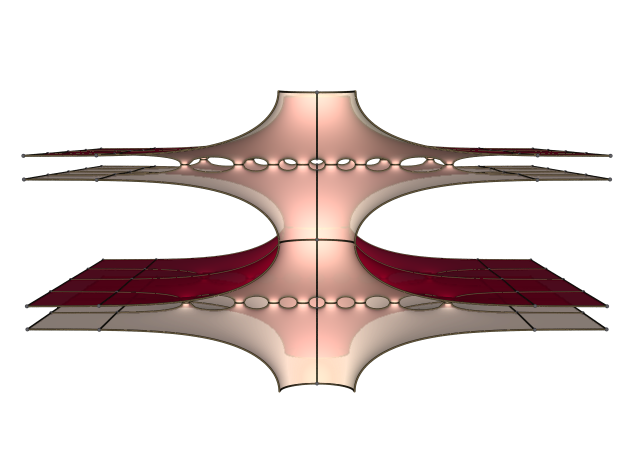}
		\includegraphics[width=3in]{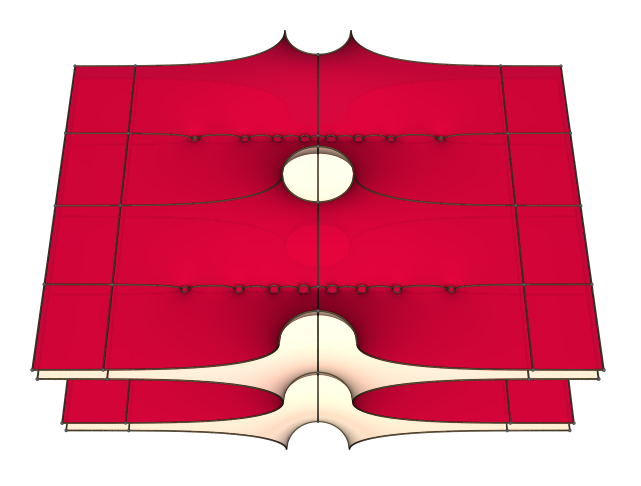}
	}
	\caption{Genus 8 surface.   The locations of the six small necks correspond to the roots of the polynomial $p_8(z)=z^8+64z^7+784z^6+3136z^5+4900z^4+3136z^3+784z^2+64z+1$.}
	\label{figure:Wei(1,8)}
\end{figure}  

\begin{proof}
In this case, the balance equations are given by the following equations.
\[
\begin{split}
F_{1,1} &=\sum_{j=1}^n\frac{1+a_j}{n(1-a_j)} \\
F_{2,i} &=\sum_{j \neq i}\frac{a_i+a_j}{n^2(a_i-a_j)}+\frac{1+a_i}{n(1-a_i)} \\
\end{split}
\]
Observe first that the polynomials $p_n$ satisfy the 
hypergeometric differential equation
\[
z(1-z)p_n''(z)+(1+(2n-1)z)p_n'(z)-n^2p_n(z)=0.
\]
In particular, all roots are simple. Furthermore,
\[
p_n(z)=z^n p_n(1/z)
\]Thus, for $n=2k$, the roots will be $a_1,\cdots,a_k,1/a_1,\cdots,1/a_k$ and for $n=2k+1$, the roots will be $a_1,\cdots,a_k,1/a_1,\cdots,1/a_k,-1$.  Hence, $F_{1,1}=0$ by symmetry.

Since $p_n$ only has simple zeroes, for each zero $a_k$ we get the following equation.
\[
p_n''(a_k)=2p_n'(a_k)\sum_{j \neq k}\frac{1}{a_k-a_j}
\]
Plugging this into the hypergeometric differential equation for $p_n$, we get that
\[
0=2a_k(1-a_k)\sum_{j \neq k}\frac{1}{a_k-a_j}+1+(2n-1)a_k.
\]
This implies easily that  $F_{2,k}=0$ for $1 \leq k \leq n$, and so the given configuration is balanced $\forall n \in \N$.

To show that the configuration is non-degenerate, let M be the matrix with entries 
\[
M_{i,j}=\frac{\partial F_{2,i}}{\partial p_{2,j}}.  
\]
Then
\[
M_{i,i}=\sum_{k \neq i}\frac{-2a_k}{n(a_i-a_k)^2}+\frac{2}{(1-a_i)^2} 
\]
and, if $i \neq j$,
\[
M_{i,j}=\frac{2a_i}{n(a_i-a_j)^2}.
\]
Thus, 
\[
\begin{split}
\sum_{i \neq j}|M_{i,j}| &=\sum_{i \neq j}\frac{-2a_i}{n(a_i-a_j)^2} \\
&=\sum_{i \neq j}\frac{-2a_i}{n(a_j-a_i)^2} \\
&=M_{j,j}-\frac{2}{(1-a_j)^2} \\
&\leq M_{j,j} \\
\end{split}
\]
for $j=1,\cdots,n$.  Hence, M is invertible and the differential of F has rank n.  Thus, this configuration is non-degenerate.
\end{proof}

\subsection{Combining non-degenerate balanced configurations}
The next proposition requires two new definitions.  They are adjustments on similar terms from \cite{tr2}.  Let $F_{k,i}^+$ be the sum of the forces exerted by the $p_{k+1,j}$ terms on $p_{k,i}$ and $F_{k,i}^-$ be the sum of the forces exterted by the $p_{k-1,j}$ terms on $p_{k,i}$, i.e.
\[
F_{k,i}^+=(-1)^k\sum_{j=1}^{n_{k+1}}\frac{p_{k+1,j}^{(-1)^k}}{n_kn_{k+1}\left(p_{k+1,j}^{(-1)^k}-p_{k,i}^{(-1)^k}\right)}
\]
and
\[
F_{k,i}^-=(-1)^{k+1}\sum_{j=1}^{n_{k-1}}\frac{p_{k,i}^{(-1)^k}}{n_kn_{k-1}\left(p_{k,i}^{(-1)^k}-p_{k-1,j}^{(-1)^k}\right)}\quad.
\]
\begin{proposition}
Let $p_{k,i}$ and $p_{k,i}'$ be two balanced configurations.  Assume that:
\begin{enumerate}
\item $n_1=n_1'=1$, 
\item $p_{1,1}=p_{1,1}'=1$, 
\item $F_{1,1}^+=F_{1,1}'^+\neq 0$.  
\end{enumerate}
Define $p_{k,i}''$ as follows:
\[
\begin{split}
& \forall k \in\{1,\cdots,N\},\,n_k''=n_k \text{ and } p_{k,i}''=p_{k,i} \\
& \forall k \in\{1,\cdots,N'\},\,n_{k+N}''=n_k' \text{ and } p_{k+N,i}''=p_{k,i}'e^T \\
& \forall k \in \mathbb{Z},\,p_{k+N+N',i}''=p_{k,i}''e^{T+T'} \\
\end{split}
\]
The configuration $p_{k,i}''$ is periodic with $N''=N+N'$ and $T''=T+T'$. Then the configuration $p_{k,i}''$ is balanced.
\label{combining}
\end{proposition}
\begin{proof}
The proof of this proposition is exactly the same as the proof of part one of proposition $3$ in \cite{tr2}.  
\end{proof}

\begin{remark}
Assuming that $p_{k,i}$ and $p_{k,i}'$ are non-degenerate balanced configurations satisfying the hypotheses of proposition \ref{combining} then, we would like to prove that $p_{k,i}''$ is also non-degenerate.  Combining this with propositions \ref{handles} and \ref{combining} would then show the existence of surfaces with an arbitrary number of ends that satisfy properties \ref{property1} and \ref{property2}.  This is quite technical, however, and we omit the proof. We will treat a special case in Proposition \ref{prop:crucial}  that allows us
to establish the existence of surfaces with arbitrarily many ends and arbitrary genus. 
\end{remark}

\begin{proposition}
Let $N=2,n_1=1,n_2=n,T=0, N'=2, n_1'=1, n_2'=m$, and $T'=0$.  Also, let $a_1,\ldots,a_n$ be the roots of the polynomial $p_n(z)=\sum_{k=0}^n{n \choose k}^2z^k$ and $b_1,\ldots,b_m$ be the roots of the polynomial $p_m(z)=\sum_{k=0}^m{m \choose k}^2z^k.$  Then there exists a non-degenerate balanced configuration $\{p_{k,i}''\}$ with $p_{1,1}''=1$, $p_{2,i}''=a_i$ for $i=1,\ldots,n$, $p_{3,1}''=1$, and $p_{4,i}''=b_i$ for $i=1,\ldots,m$.
\label{8ends}
\end{proposition}

\begin{figure}[H]
	\centerline{ 
		\includegraphics[width=2in]{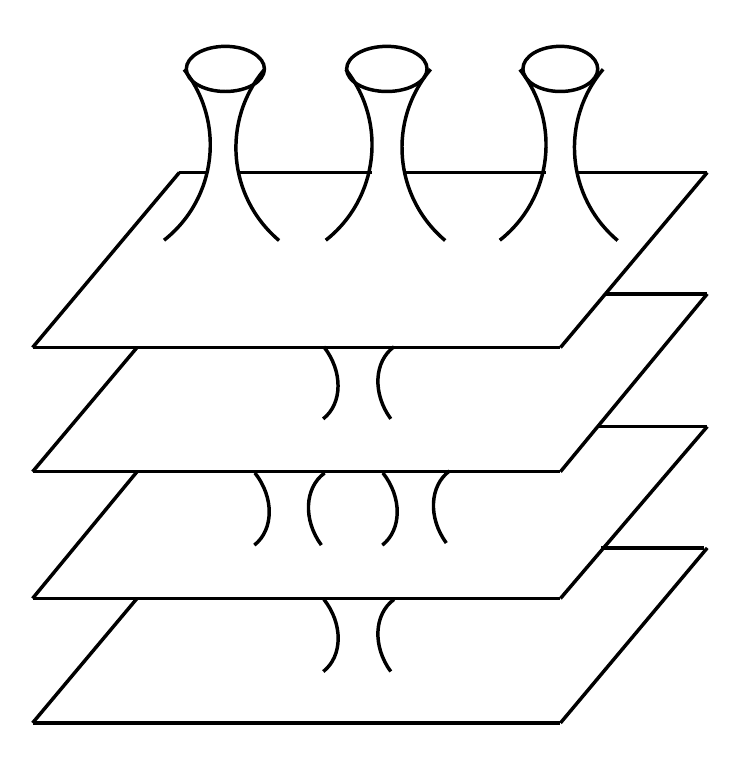}
	}
	\caption{Surface corresponding to non-degenerate balanced configuration with $n=2$, and $m=3$.
}
	\label{(1,2,1,3)setup}
\end{figure}  

\begin{proof}
Let $p_{1,1}=p_{1,1}'',p_{2,i}=p_{2,i}'',p_{1,1}'=p_{3,1}''$, and $p_{2,i}'=p_{4,i}''$.  Then $n_1=n_1'=1$ and $p_{1,1}=p_{1,1}'=1$.  Also,
\[
F_{1,1}^+=-\sum_{j=1}^n\frac{1}{n(1-a_j)}
\]
and
\[
F_{1,1}'^+=-\sum_{j=1}^n\frac{1}{n(1-b_j)}.
\]
If $n$ is even then order the roots of $p_n$ such that $a_{n/2+k}=1/a_k$ for $k=1,\ldots,n/2$.  Then, after a brief computation,
\[
F_{1,1}^+=-\frac{1}{2}.
\]


If $n$ is odd then order the roots of $p_n$ such that $a_{(n-1)/2+k}=1/a_k$ for $k=1,\ldots,(n-1)/2$ and $a_n=-1$.  Then
\[
F_{1,1}^+=-\frac{1}{2}
\]
Thus, $F_{1,1}^+=-\frac{1}{2}$ and, similarly, $F_{1,1}'^+=-\frac{1}{2}.$  Hence, the hypotheses of proposition \ref{combining} are met.  Therefore, $\{p_{k,i}''\}$ is a balanced configuration.  Since we didn't really prove the non-degeneracy portion of proposition \ref{combining}, we can prove that directly for this balanced configuration.    

Let $N$ be the matrix with entries $$N_{i,j}=\frac{\partial F_{2,i}''}{\partial p_{2,j}''}$$ and $M$ be the matrix with entries $$M_{i,j}=\frac{\partial F_{4,i}''}{\partial p_{4,j}''}.$$  As shown in proposition \ref{handles}, $M$ and $N$ are invertible.  Also, let
\[
\alpha=\sum_{j=1}^m\frac{p_{4,j}}{m(p_{4,j}-1)^2}+\sum_{j=1}^n\frac{p_{2,j}}{n(p_{2,j}-1)^2},
\]
\[
\beta=\left(\frac{-1}{n(p_{2,1}-1)^2},\ldots,\frac{-1}{n(p_{2,n}-1)^2}\right),
\]
and
\[
\gamma=\left(\frac{-1}{m(p_{4,1}-1)^2},\ldots,\frac{-1}{m(p_{4,m}-1)^2}\right).
\]

Then $DF_{1,1}''=\left(\alpha,\beta,0,\gamma\right)$ and $DF_{3,1}''=\left(0,\beta,\alpha,\gamma\right)$.  Therefore,
\[
DF''=\begin{bmatrix}
\alpha & \beta & 0 & \gamma \\ \cdot & N & \cdot & 0 \\ 0 & \beta & \alpha & \gamma \\ \cdot & 0 & \cdot & M  
\end{bmatrix}
\]
and $\text{rank}\left(DF''\right)\geq n+m+1$.  Since the sum of forces is always zero, $DF''$ can't have full rank.  Thus, $\text{rank}\left(DF''\right)=n+m+1$, and so $\{p_{k,i}''\}$ is a non-degenerate balanced configuration.
\end{proof}

\newpage
\begin{proposition}\label{prop:crucial}
Let $N\in\mathbb{N},n_k=1,p_{k,1}=(-1)^{k+1}$ for $k=1,\ldots,N$, $T=0$, $N'=2,n_1'=1,n_2'=n\in\mathbb{N},p_{1,1}'=1$, $p_{2,i}'=a_i$ where $a_1,\ldots,a_n$ are the distinct real roots of the polynomial $p_n(z)=\sum_{k=0}^n{n \choose k}^2z^k$, and $T'=0$..  Also, let $N''=N+N'=N+2$, $n_k''=1$ for $k=1,\ldots,N+1$, $n_N''=n$, $p_{1,k}''=(-1)^{k+1}$ for $k=1,\ldots,N+1$, $p_{N+2,i}''=a_i$ for $i=1,\ldots,n$, and $T''=T+T'=0$.  Then $\{p_{k,i}''\}$ is a non-degenerate balanced configuration.
\label{many ends}
\end{proposition}
\begin{figure}[H]
	\centerline{ 
		\includegraphics[width=2in]{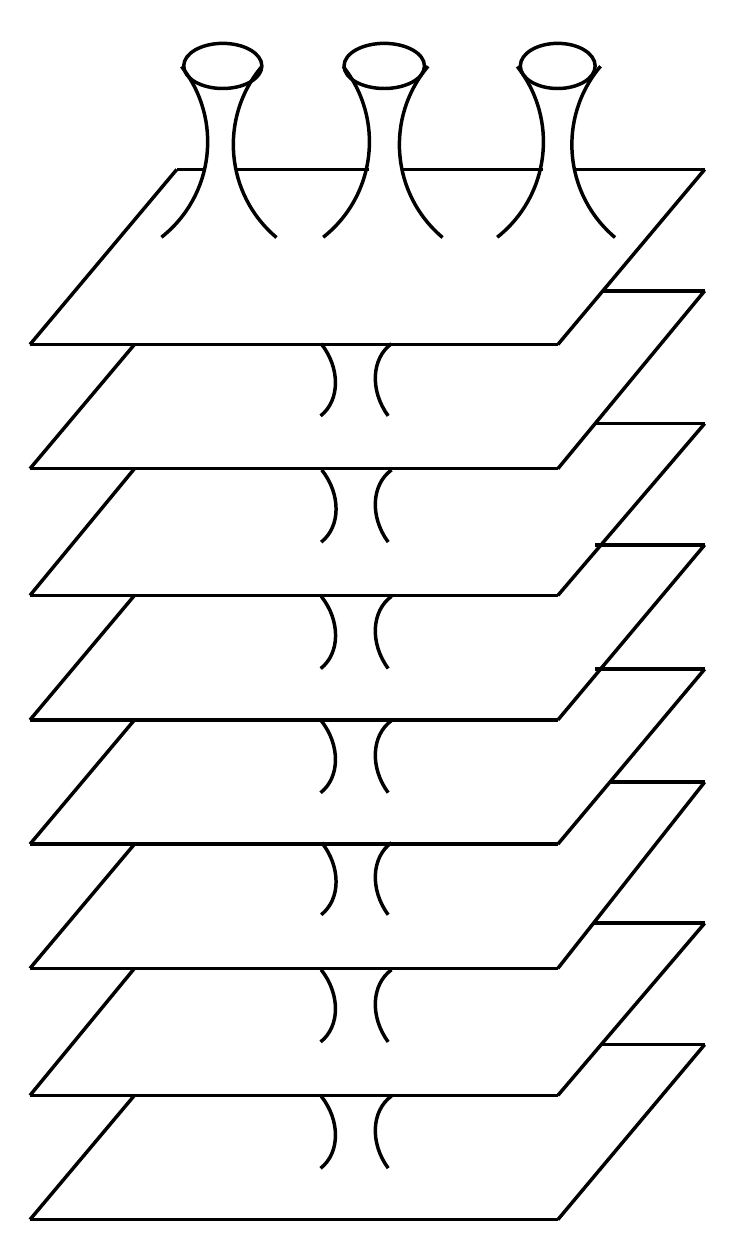}
	}
	\caption{Surface corresponding to non-degenerate balanced configuration with $n_k''=1$ for $k=1,\ldots,7$ and $n_8''=3$.
}
	\label{(1,1,1,1,1,1,1,3)}
\end{figure}  

\begin{proof}
First, we need to show that $\{p_{k,i}\}$ is a balanced configuration:  
\[
F_{2k,1}=\frac{p_{2k+1,1}}{p_{2k+1,1}-p_{2k,1}}-\frac{p_{2k,1}}{p_{2k,1}-p_{2k-1,1}}=\frac{1}{2}-\frac{-1}{-2}=0
\]
and
\[
F_{2k+1,1}=\frac{p_{2k,1}}{p_{2k,1}-p_{2k+1,1}}-\frac{p_{2k+1}}{p_{2k+1}-p_{2k+2}}=\frac{-1}{-2}-\frac{1}{2}=0.
\]
That $\{p_{k,i}'\}$ is a balanced configuration follows from proposition \ref{handles}.  Now, 
\[
F_{1,1}^+=-\frac{p_{1,1}}{p_{1,1}-p_{2,1}}=-\frac{1}{2}
\]
and, similar to the proof of proposition \ref{8ends}, $F_{1,1}'^+=-\frac{1}{2}$.  Thus, by proposition \ref{combining}, $\{p_{k,i}''\}$ is a balanced configuration.

As far as the non-degeneracy, let's first write out the forces $F_{k,i}''$:
\[
F_{1,1}''=\sum_{j=1}^n\frac{p_{n,j}''}{n\left(p_{n,j}''-p_{1,1}''\right)}-\frac{p_{1,1}''}{p_{1,1}''-p_{2,1}''};
\]
\[
F_{k,1}''=
\begin{cases}
\frac{p_{k+1,1}''}{p_{k+1,1}''-p_{k,1}''}-\frac{p_{k,1}''}{p_{k,1}''-p_{k-1,1}''}\text{, $k$ even},
\\
\frac{p_{k-1,1}''}{p_{k-1,1}''-p_{k,1}''}-\frac{p_{k,1}''}{p_{k,1}''-p_{k+1,1}''}\text{, $k$ odd}.
\end{cases}
\]
for $k=2,\ldots,N$;
\[
F_{N+1}''=\frac{p_{N,1}''}{p_{N,1}''-p_{N+1,1}''}-\sum_{j=1}^n\frac{p_{N+1,1}''}{n\left(p_{N+1,1}''-p_{N+2,j}''\right)};
\]
and
\[
F_{N+2,i}''=\sum_{j\neq i}\frac{p_{N+2,i}''+p_{N+2,j}''}{n^2\left(p_{N+2,i}''-p_{N+2,j}''\right)}+\frac{p_{1,1}''}{n\left(p_{1,1}''-p_{N+2,i}''\right)}-\frac{p_{N+2,i}''}{n\left(p_{N+2,i}''-p_{N+1,j}\right)}.
\]

Let $M$ be the $n\times n$ matrix with entries
\[
M_{i,j}=\frac{\partial F_{N+2,i}''}{\partial p_{2,j}''}.
\]
As shown in the proof of proposition \ref{handles}, $M$ is invertible.

Let $N$ be the $N\times N$ matrix with entries
\[
N_{i,j}=\frac{\partial F_{i+1,1}}{\partial p_{j,1}''}.
\]
If $k\in\{2,\ldots,N\}$ then 
\[
\frac{\partial F_{k,1}''}{\partial p_{k-1,1}''}=\frac{-1}{4}, \frac{\partial F_{k,1}''}{\partial p_{k+1,1}''}=\frac{1}{4}, \text{ and }\frac{\partial F_{k,1}''}{\partial p_{j,1}''}=0 \text{ if }j\neq k-1,k+1.
\]

Also,
\[
\frac{\partial F_{N+1,1}''}{\partial p_{k,1}''}=0 \text{ if }k<N+1 \text{ and }\frac{\partial F_{N+1,1}''}{\partial p_{N,1}''}=\frac{-1}{4}.
\]

Hence, $N_{i,i}=-\frac{1}{4}$ for $i=1,\ldots,N$ and $N_{i,j}=0$ if $j<i$, and so $N$ is an invertible matrix.  Therefore,
\[
DF''=\begin{bmatrix}
\cdot & \cdot & \cdot \\ N & \cdot & P \\ Q & \cdot & M
\end{bmatrix}
\]
where $P$ is the $N\times n$ matrix with $P_{k,i}=0$ if $k=1,\ldots,N-1$ and $P_{N,i}=\frac{1}{n(a_i-1)^2}$ and $Q$ is the $n\times N$ matrix with $Q_{k,1}=-\frac{a_k}{n(a_k-1)^2}$ for $k=1,\ldots,n$ and $Q_{k,i}=0$ if $k=2,\ldots,N$.  Thus, $DF''$ has rank $N+n$ and $\{p_{k,i}''\}$ is non-degenerate.

\end{proof}

\subsection{Other examples and non-examples}
\begin{proposition}
There does not exist a balanced configuration $\{p_{k,i}\}$ with $N=2,n_1=n_2=2$, and $T=0$.
\end{proposition}
\begin{proof}
Using Mathematica to solve the balance equations, we found that the only possible solution in which the $p_{k,i}$ are distinct is $\{p_{1,1},p_{1,2},p_{2,1},p_{2,2}\}=\{1,-1,I,-I\}$.  However, this is the same as the balanced configuration with $N=4$, $n_1=n_2=n_3=n_4=1$, and $T=0$.  
\end{proof}
\begin{proposition}
There exists a non-degenerate balanced configuration $\{p_{k,i}\}$ with $N=2,n_1=2,n_2=3$, and $T=0$.
\label{Wei(2,3)}
\end{proposition}
\begin{figure}[H]
	\centerline{ 
		\includegraphics[width=3.3in]{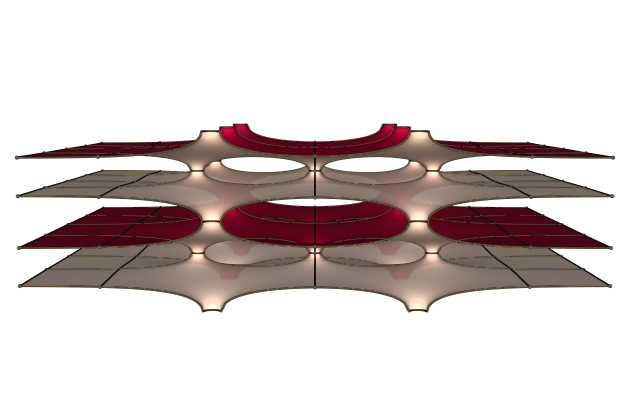}
		\includegraphics[width=2.7in]{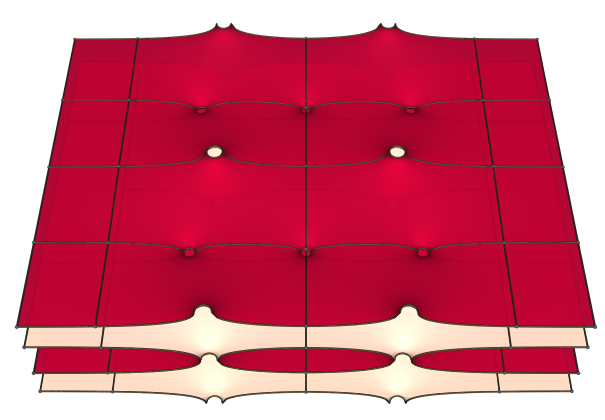}
	}
	\caption{Side and top views of a genus four surface with $N=2,n_1=2,n_2=3$.  
}
	\label{figure:Wei(2,3)}
\end{figure}  
\begin{proof}
The force equations corresponding to this setup are
\[
F_{1,i}=(-1)^i\frac{p_{1,1}+p_{1,2}}{4\left(p_{1,2}-p_{1,1}\right)}-\sum_{j=1}^3\frac{p_{1,i}+p_{2,j}}{6\left(p_{1,i}-p_{2,j}\right)}
\]
for $i=1,2$ and
\[
F_{2,i}=\sum_{j\neq i}\frac{p_{2,i}+p_{2,j}}{9\left(p_{2,i}-p_{2,j}\right)}+\sum_{j=1}^2\frac{p_{1,j}+p_{2,i}}{6\left(p_{1,j}-p_{2,i}\right)}
\]
for $i=1,2,3$.
Let $a_1=4+2\sqrt{5}+\sqrt{35+16\sqrt{5}}$ and $a_2=\frac{1}{2}\left(-17-9\sqrt{5}-\sqrt{690+306\sqrt{5}}\right)$, and let $p_{1,1}=a_1,p_{1,2}=1/a_1,p_{2,1}=a_2,p_{2,2}=-1$, and $p_{2,3}=1/a_2$.  Then $F_{k,i}=0$ for $k=1,2$ and $i=1,\ldots,n_k$, and so $\{p_{k,i}\}$ is a balanced configuration.  

Elementary row operations show that $DF$ row reduces to 
\[
DF=\begin{bmatrix}
1 & 0 & 0 & 0 & \cdot \\ 0 & 1 & 0 & 0 & \cdot \\ 0 & 0 & 1 & 0 & \cdot \\ 0 & 0 & 0 & 1 & \cdot \\ 0 & 0 & 0 & 0 & 0
\end{bmatrix}
\approx
\begin{bmatrix}
1 & 0 & 0 & 0 & 626.396 \\ 0 & 1 & 0 & 0 & 2.19707 \\ 0 & 0 & 1 & 0 & -1376.24 \\ 0 & 0 & 0 & 1 & -37.0977 \\ 0 & 0 & 0 & 0 & 0
\end{bmatrix}.
\]
Therefore, $\{p_{k,i}\}$ is a non-degenerate balanced configuration.
\end{proof}

Numerical evidence suggests:
\begin{conjecture}
There exists a non-degenerate balanced configuration $\{p_{k,i}\}$ with $N=2,n_1=2,n_2=2k-1$, and $T=0$ for $k\in\mathbb{N}$.
\end{conjecture}

\section{Weierstrass Data}
We begin the proof of Theorem $1$ by parametrizing a set of Riemann surfaces and Weierstrass data that are candidates for the minimal surfaces we want to construct.  The construction is almost exactly the same as in \cite{tr2}.  The main difference is our definition of the Gauss map $G$.  We repeat the details for the convenience of the reader.

Let ${\bar \C}_k=\overline{\mathbb{C}}$ for $k=1,\ldots,N$, and for each $k\in\{1,\ldots,N\}$ let $G_k:{\bar \C}_k:\mapsto{\bar \C}$ be the meromorphic function defined by
\[
G_k(z)=\delta_k z\left(
\sum_{i=1}^{n_k} \frac{\alpha_{k,i}}{z-a_{k,i}} - 
\sum_{i=1}^{n_{k-1}} \frac{\beta_{k,i}}{z-b_{k,i}}
\right)
\]
where $\delta_k\in(0,\infty)$, the poles $a_{k,i}$ and $b_{k,i}$ are distinct non-zero complex numbers, and the $\alpha_{k,i}$ and $\beta_{k,i}$ are non-zero complex numbers such that
\[
\sum_{i=1}^{n_k} \alpha_{k,i}=
\sum_{i=1}^{n_{k-1}} \beta_{k,i}
=1.
\]
The first equality ensures that $G_k(z)$ has a zero at $\infty$. The zeroes at $0$ and $\infty$ are needed to ensure that the Gauss map is vertical at the annular ends.  The $\delta_k$ terms will be used to ensure that the periods at the ends are the same.  In \cite{tr2}, the corresponding map is $g_k(z)=\frac{G_k(z)}{\delta_k z}$.

Let $\alpha_k=(\alpha_{k,1},\ldots,\alpha_{k,n_k})$ and $\alpha=(\alpha_1,\ldots,\alpha_N)$, and define $\beta,\gamma,a$ and $b$ in the same way. Let $\delta=(\delta_1,\ldots,\delta_N)$ and $X=(\alpha,\beta,\delta,\gamma,a,b)$.  The set $X$ is our parameter space used to construct the Riemann surfaces and Weierstrass data. Within this space,  we will solve the period problem.

The surfaces we are constructing have $n_k$ catenoid-shaped necks between the $k$ and $k+1$ levels.  In oder to achieve this, we use the $G_k$ functions to create coordinates near each pole and identify an annulus centered at $a_{k,i}\in{\bar \C}_k$ with an annulus centered at $b_{k+1,i}\in{\bar \C}$ for $k=1\ldots,N$ and $i=1,\ldots,n_k$ using the following procedure.

The function $v_{k,i} = 1/G_k$ has a simple zero at $a_{k,i}$.  Thus, there exists $\epsilon>0$ such that $v_{k,i}$ is a biholomorphic map from a neighborhood of $a_{k,i}\in{\bar \C}_k$ to the disk $D_\epsilon(0)$. In this manner, $v=v_{k,i}$ is a complex coordinate in a neighborhood of $a_{k,i}$. Similarly, $w=w_{k+1,i}=1/G_{k+1}$ is a biholomorphic map from a neighborhood of $b_{k+1,i}\in{\bar \C}_{k+1}$ to the disk  $D_\epsilon(0)$.  Thus, for each pair $a_{k,i}$ and $b_{k+1,i}$ we get the pair of coordinates $v=v_{k,i}$ and $w=w_{k+1,i}$. 

\begin{figure}[htp]
\begin{center}
\includegraphics[width=3.5in]{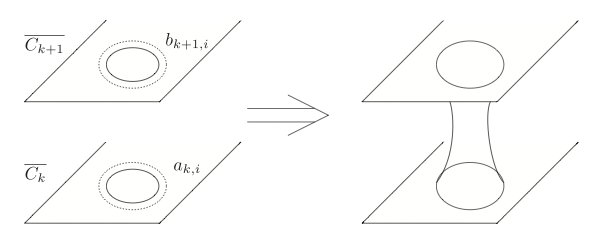}
\caption{Gluing construction.}
\end{center}
\end{figure}

Choose a complex gluing parameter $r$ with  parameter $|r|\in (0,\epsilon^2)$ and remove the disks $|v|\le \frac{|r|}\epsilon$ and $|w|\le \frac{|r|}\epsilon$ from ${\bar \C}_k$ and ${\bar \C}_{k+1}$, respectively. Then, we create a conformal model of the catenoid-shaped neck by identifying the points in ${\bar \C}_{k}$ 
satisfying
\[
\frac {|r|} \epsilon<|v|<\epsilon
\]
with points in ${\bar \C}_{k+1}$ satisfying
\[
\frac {|r|} \epsilon<|w|<\epsilon
\]
by the equation
\[
v w = r.
\]

Let $\Sigma$ be the compact Riemann surface created by repeating this procedure for each $k=1,\ldots,N$ and $i=1,\ldots,n_k$.  Denote by $\Sigma^*$ the surface obtained by removing the points $0_k$ and $\infty_k$ from $\Sigma$ for all $k$.  When $r=0$, define $\Sigma$ as the disjoint union ${\bar \C}_1\cup{\bar \C}_2\cup\ldots\cup{\bar \C}_N$. This is the underlying Riemann surface for our minimal surface candidates.

Next, the Gauss map $G:\Sigma\to\bar {\C}$ is defined by
\begin{equation}
G(z)=
\begin{cases}
\sqrt r G_k(z) & \text{if $z\in {\bar \C}_k$, $k$ even},
\\
\frac1{\sqrt{r} G_k(z)} & \text{if $z\in {\bar \C}_k$, $k$ odd}.
\end{cases}
\end{equation}

If $k$ is even, then $G=\sqrt r/v$ on ${\bar\C}_k$ and 
$G=w/\sqrt r$ on ${\bar\C}_{k+1}$.  If $k$ is odd, then $G=v/\sqrt{r}$ on ${\bar \C}_k$ and $G=\sqrt{r}/w$ on ${\bar \C}_{k+1}$.  Therefore, the relation $vw=r$ implies that $G$ is well-defined on $\Sigma$.

Before defining our height differential $\eta$, we need to choose a basis of the homology of $\Sigma$.  Define $A_{k,i}$ to be the circle $|v_{k,i}|=\epsilon$ in ${\bar \C}_k$ oriented positively. The construction of $\Sigma$ implies that this is homotopic to the circle $\abs{w_{k+1,i}}=\epsilon$ oriented negatively.
Choose $B_{k,i}$, $i\geq2$, to be a closed curve in $\Sigma$ such that $A_{k,1}\cdot B_{k,i}=-1$, $A_{k,i}\cdot B_{k,i}=1$, $A_{m,n}\cdot B_{k,i}=0$ if $m\neq k$, and $B_{k,i}\cdot B_{m,n}=0$ if $(m,n)\neq(k,i)$.  Finally, choose $B_{1,1}$ to be a closed curve such that $A_{k,1}\cdot B_{1,1}=1$ for $k=1,\ldots,N$ and it doesn't intersect any of the above curves.  Then a basis of $H_1(\Sigma)$ is given by the curves $A_{1,1}, B_{1,1}, A_{k,i}$, and $B_{k,i}$, with $k=1,\ldots,N$ and $i=2,\ldots,n_k$.  Note that if we replace the $B_{1,i}$ curves by $B_{1,i}'=B_{1,i}+B_{1,1}$ then we get a canonical basis of $H_1(\Sigma)$.

\begin{figure}[H]
\begin{center}
\includegraphics[width=3in]{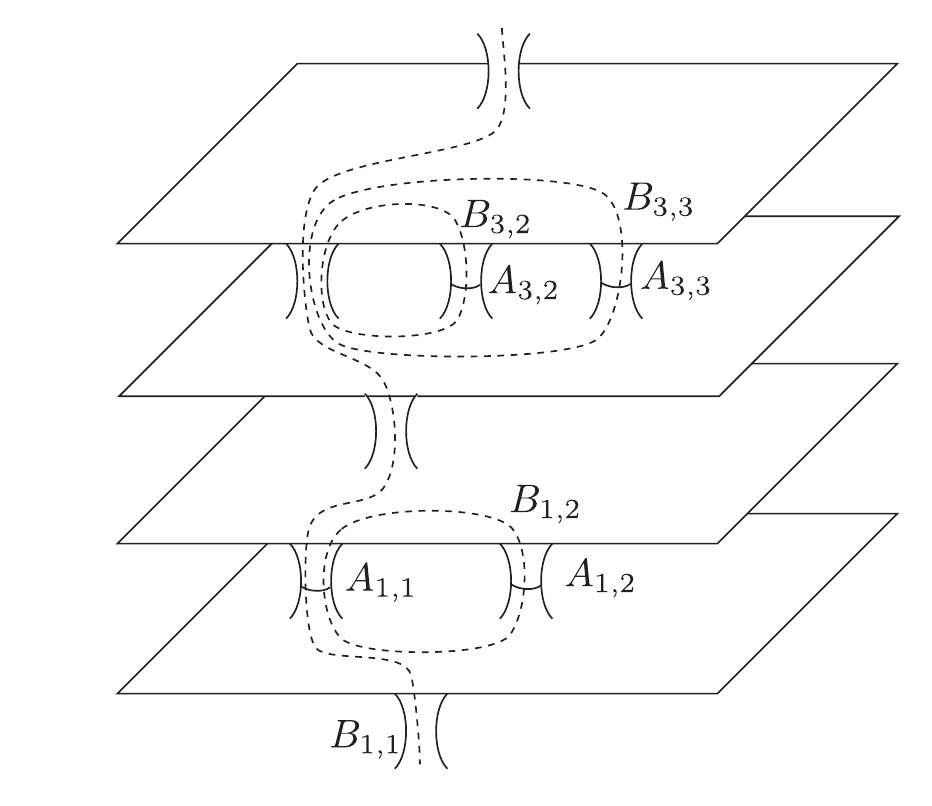}
\caption{Labeling the cycles.}
\end{center}
\end{figure}

\begin{proposition}[\cite{tr2}]
Consider numbers $\gamma_{k,i}$, $k=1,\ldots,N$, $i=1,\ldots,n_k$ such that for any $k$,
\[
\sum_{i=1}^{n_k} \gamma_{k,i} = 1.
\]
Then there exists a unique holomorphic 1-form $\eta$ on $\Sigma$ such that for any $k=1,\ldots,N$, $i=1,\ldots,n_k$,
\[
\int_{A_{k,i}} \eta = 2\pi i \gamma_{k,i}
\]
\end{proposition}

The proof is the same as in the proof of proposition $5$ from section $3.3$ in \cite{tr2}. 

We now have a space of Riemann surfaces and Weierstrass data that are candidates for the surfaces we want to construct.  The parameters are given by $(r,X))$, and we will look at what happens when $r\rightarrow 0$.

\section{Constraints on the Weierstrass data and period conditions}
We express the Weierstrass data using the notation $$\psi(z)=\text{Re}\int_{z_0}^z\left(\phi_1,\phi_2,\phi_3\right)$$ where $z_0\in\Sigma$ is a base point, $\phi_1=\frac{1}{2}\left(\frac{1}{G}-G\right)\eta$, $\phi_2=\frac{i}{2}\left(\frac{1}{G}+G\right)\eta$, and $\phi_3=\eta$.  In order that $(\Sigma,G,\eta)$ are the Weierstrass data of a complete, doubly periodic minimal surface with horizontal embedded ends, we need:

\begin{enumerate}
\item
For any $p\in\Sigma^*$, $\eta$ has a zero at $p$ if and only of $G$ has either a zero or pole at $p$, with the same
multiplicity. At each puncture $0_k$ and $\infty_k$, $G$ has a zero or pole of order $n\geq1$ and $\eta$ has a zero of order $n-1$.
\item
For any closed curve $c$ om $\Sigma^*$, $\re \int_c \phi_j$ is an integral linear combination of two linearly independent vectors of $\R^3$. We denote the set of these linear combinations by $\Lambda$.
\end{enumerate}

As the zeroes and poles of $G$ are the zeroes of the $G_k$, we can write condition (1) equivalently as
\begin{enumerate}
\item[1'.]
The zeroes of $\eta$ are the zeroes of $G_k dz/z$, $k=1,\ldots,N$, with the same multiplicity.
\end{enumerate}

If condition $(1)$ is satisfied then the 1-forms $\phi_1$ and $\phi_2$ have poles only at the punctures of $\Sigma^*$, and so condition $(2)$ needs to be checked only for a canonical basis of the homology of $\Sigma$ and for small loops around the punctures.  Therefore we can rewrite the condition (2) as follows:
Write $\phi=(\phi_1,\phi_2,\phi_3)$.

\begin{enumerate}
\item[2'.1] For any $k=1,\ldots,N$ and $i=1,\ldots,n_k$, 
\[
\re \int_{A_{k,i}} \phi= 0
\]
\item[2'.2]
For any   $k=1,\ldots,N$ and $i=2,\ldots,n_k$, 
\[
\re \int_{B_{k,i}} \phi = 0
\]
\item[2'.3]
\[
\re \int_{B_{1,1}} \phi \in \Lambda
\]
\item[2'.4] For any $k=1,\ldots,N$,
\[
\re \int_{\partial D_\epsilon(0_k)} \phi \in \Lambda
\]
\item[2'.5] For any $k=1,\ldots,N$,
\[
\re \int_{\partial D_\epsilon(\infty_k)} \phi \in \Lambda
\]
\end{enumerate}
If $\re \int_{A_{k,i}} \phi= 0$ and $\re \int_{\partial D_\epsilon(0_k)} \phi \in \Lambda$ for each $k,i$ then the period condition at $\infty_k$ is automatically satisfied by Cauchy's theorem.
Observe that the period vectors $\re \int_{\partial D_\epsilon(0_k)} \phi $ and $\re \int_{\partial D_\epsilon(\infty_k)} \phi $ are necessarily horizontal, as $\eta=\phi_3$ is holomorphic at $0_k$ and $\infty_k$.

\section{Height differential extends holomorphically to $r=0$}

This section follows directly from \cite{tr2}.  Recall that when $r=0$, we defined $\Sigma$ as the disjoint union ${\bar \C}_1\cup
{\bar \C}_2\cup\ldots\cup{\bar \C}_N$.  The Gauss map is defined  when $r=0$ and depends holomorphically on $r$.  We need the same to be true for the height differential.  When $r=0$, define $\eta$ by $\eta_k\circ{\bar \C}_k$ where $\eta_k$ is the unique meromorphic 1-form on ${\bar \C}_k$ with simple poles at $a_{k,i}$ and $b_{k,i}$ with residues $\gamma_{k,i}$ and $-\gamma_{k-1,i}$, i.e.
\[
\eta_k=
\left(\sum_{i=1}^{n_k}\frac{\gamma_{k,i}}{z-a_{k,i}} - 
\sum_{i=1}^{n_{k-1}}\frac{\gamma_{k-1,i}}{z-b_{k,i}}
\right)\, dz.
\]
Observe that our conditions ensure that $\eta$ is holomorphic at $0_k$ and $\infty_k$ for each $k$.

The next two propositions are from section $4$ in \cite{tr2}.  As our height differential is defined in the same way as in \cite{tr2}, the proofs of these propositions are the same.
\begin{proposition}[\cite{tr2}]
Let $z\in\overline{\C}_k,\,z\neq a_{k,i},\,z\neq b_{k,i}.$  Then $r\mapsto \eta(z)$ is holomorphic in a neighborhood of $0$.
\label{prop holo1}
\end{proposition}
\begin{proposition}[\cite{tr2}]
Let $v=v_{k,i}$.  On the domain $\frac{\abs{r}}{\epsilon}<\abs{v}<\epsilon$ of $\Sigma$, we have the formula
\[
\eta=f\left(v,\frac{r}{v}\right)\frac{dv}{v}=-f\left(\frac{r}{w},w\right)\frac{dw}{w}
\]
where $f$ is a holomorphic function of two complex variables defined in a neighborhood of $(0,0)$.
\label{prop holo2}
\end{proposition}

We can use propositions \ref{prop holo1} and \ref{prop holo2} to estimate the integrals of $\eta,G\eta$, and $1/G\eta$ on the homology cycles and on cycles around the punctures.  These are necessary to solve the period problem when $r=0$.  As in \cite{tr2}, we will use a term $\text{holo}(r,X)$, meaning a holomorphic function in terms of $(r,X)$ in a neighborhood of $(0,X_0)$.  
\begin{proposition}[\cite{tr2}]
\[
\begin{split}
\int_{A_{k,i}}G^{(-1)^k}\eta &=\sqrt{r}\left(2\pi i \text{ res}_{a_{k,i}}G_k\eta_k+\text{$r$ holo}(r,X)\right) \\
\int_{A_{k,i}}G^{(-1)^{k+1}}\eta &=\sqrt{r}\left(-2\pi i \text{ res}_{b_{k+1,i}}G_{k+1}\eta_{k+1}+\text{$r$ holo}(r,X)\right)  \\
\int_{B_{k,i}}\eta &=\left(\gamma_{k,i}-\gamma_{k,1}\right)\text{log$(r)$}+\text{holo}(r,X) \\
\int_{B_{k,i}}G^{(-1)^k}\eta &=\frac{1}{\sqrt{r}}\left(\int_{b_{k+1,i}}^{b_{k+1,1}}G_{k+1}^{-1}\eta_{k+1}+r\text{ log$(r)$ holo$(r,X)+r$ holo$(r,X)$}\right) \\
\int_{B_{k,i}}G^{(-1)^{k+1}}\eta &=\frac{1}{\sqrt{r}}\left(\int_{a_{k,1}}^{a_{k,i}}G_{k}^{-1}\eta_{k}+r\text{ log$(r)$ holo$(r,X)+r$ holo$(r,X)$}\right) \\
\end{split}
\]
The proofs are the same as in \cite{tr2}, section 5. The following proposition takes care of the different nature of our annular ends compared to the planar ends on \cite{tr2}.
\label{prop expansion1}
\end{proposition}

\begin{proposition}
\[
\begin{split}
\int_{\partial D_\epsilon(0_k)}G^{(-1)^k}\eta &=0 \\
\int_{\partial D_\epsilon(0_k)}G^{(-1)^{k+1}}\eta &=\frac{1}{\sqrt{r}}\left(2\pi i \text{ res}_{0}\frac{1}{G_k}\eta_{k}+\text{$r$ holo}(r,X)\right)  \\
\end{split}
\]
\label{prop expansion2}
\end{proposition}
\begin{proof}
First,
\[
\int_{\partial D_\epsilon(0_k)}G^{(-1)^k}\eta=\sqrt{r}\int_{\partial D_\epsilon(0_k)}G_k\eta=0
\]
because $G_k\eta$ has no poles in a neighborhood of $0_k$.  Using proposition \ref{prop holo1},
\[
\begin{split}
\int_{\partial D_\epsilon(0_k)}G^{(-1)^{k+1}}\eta=&\frac{1}{\sqrt{r}}\int_{\partial D_\epsilon(0_k)}\frac{1}{G_k}\left(\eta_k+r\text{ holo}(r,X)dz\right)\\
=&\frac{1}{\sqrt{r}}\left(2\pi i\text{ res}_{0_k}\frac{1}{G_k}\eta_k+r\text{ holo}(r,X)\right)
\end{split}
\]
\end{proof}

\section{Solving the period problem}
We can attempt to solve the constraints on the Weierstrass data and the period problem by adjusting the variables $(r,X)$, and we will express this with a map $\mathcal{F}$.  In fact, we will find solutions when $r=0$.  This allows us to take advantage of the asymptotic expansion of each of the periods at $r=0$.  

Let $\zeta_{k,i}$ be the zeroes of $\frac{1}{\delta_k z}G_kdz$ in $\overline{\C}_k,\; i=1,\cdots,n_k+n_{k-1}-2$.  Define
\[ \mathcal{F}_{1,k,i}=\eta\left(\zeta_{k,i}\right)\text{.}\]
Abbreviate $\mathcal{F}_{1,i}=\left(\mathcal{F}_{1,k,1},\cdots,\mathcal{F}_{1,k,n_k+n_{k-1}-2}\right)$ and $\mathcal{ F}_1=\left(\mathcal{F}_{1,1},\cdots,\mathcal{F}_{1,N}\right)$.  The zeroes of $\frac{1}{\delta_k z}G_kdz$ can be thought of as the zeroes of a polynomial, and for now let's assume that they are all simple zeroes.  Section 9 in \cite{tr2} takes care of the case where $\frac{1}{\delta_kz}G_kdz$ may not only have simple zeroes, and applies here as well.  As argued in \cite{tr2}, the simple zeroes of a polynomial depend analytically on its coefficients and, by proposition \ref{prop holo1}, $\mathcal{F}_1$ depends analytically on $\left(r,X\right)$.  

If $\mathcal{F}_1=0$ then $\eta$ has at least a simple zero at each zero of $G_k$.  All the zeroes of $\frac{1}{\delta_k z}G_kdz$ are assumed to be simple, and so $G$ has $$\sum_{k=1}^N(n_k+n_{k-1}-2)=2\sum_{k=1}^Nn_k-2N$$ zeroes and poles, counting multiplicity. 

 The number of zeroes of $\eta$ is $$2\text{ genus}(\Sigma)-2=2\left(1+\sum_{k=1}^N(n_k-1)\right)-2=2+2\sum_{k=1}^Nn_k-2N-2=2\sum_{k=1}^Nn_k-2N.$$  Thus, the zeroes of $\eta$ are precisely the $\zeta_{k,i}$.

The remaining components of the map $\mathcal{F}$ deal with the period problem.  The period condition $\text{Re}\int_{A_{k,i}}\eta=0$ is taken care of by letting $\gamma_{k,i}\in\mathbb{R}$.  This is simply due to how we defined $\eta$.  From this moment on, assume that $\gamma_{k,i}\in\mathbb{R}$.   Recall that
\[
\text{Re}\int\phi_1+i\text{Re}\int\phi_2=\frac{1}{2}\left(\overline{\int G^{-1}\eta}-\int G\eta\right).
\]
With this equivalency in mind, we define
\[
\begin{split}
\mathcal{F}_{2,k,i} &=\frac{1}{\log{r}}\text{Re}\int_{B_{k,i}}\eta,\;\; i=2,\cdots,n_k. \\
\mathcal{F}_{3,k,i} &=\sqrt{r}\overline{\left(\overline{\int_{B_{k,i}}G^{-1}\eta}-\int_{B_{k,i}}G\eta\right)},\;\; i=2,\cdots,n_k. \\
\mathcal{F}_{4,k,i} &=\frac{(-1)^k}{\sqrt{r}}\left(\overline{\int_{A_{k,i}}G^{-1}\eta}-\int_{A_{k,i}}G\eta\right),\;\; i=1,\cdots,n_k. \\
\mathcal{F}_{5,k} &=(-1)^{k+1}\sqrt{r}\text{ \conj}^k\left(\overline{\int_{\partial D_{\epsilon}(0_k)}G^{-1}\eta}-\int_{\partial D_{\epsilon}(0_k)}G\eta\right)\\
\end{split}
\]

Here, $\conj$ denotes the conjugation in $\C$. 

Define the vectors $\mathcal{ F}_2$, $\mathcal{F}_3$, and $\mathcal{F}_4$ as we defined $\mathcal{F}_1$.  Let $\mathcal{F}_5=\left(\mathcal{F}_{5,1},\mathcal{F}_{5,2},...,\mathcal{F}_{5,N}\right)$ and $\mathcal{F}=\left(\mathcal{F}_1,\mathcal{F}_2,\mathcal{F}_3,\mathcal{F}_4,\mathcal{F}_5\right)$.  Note that the constraints of the Weierstrass data and the period problem listed in section 3.2 are equivalent to $\mathcal{F}=0$.  Also, there is no need for $\mathcal{F}_5$ in \cite{tr2}.

The $\log{r}$ terms that show up in $\mathcal{F}$ require us to express the variable $r$ in terms of the variable $t$ using the equation $r(t)=e^{-1/t^2}$ if $t\in\mathbb{R}\setminus\{0\}$ and $r(0)=0$.  Otherwise, the map $\mathcal{F}$ won't be differentiable at $r=0$.  Propositions \ref{prop expansion1} and \ref{prop expansion2} imply that $\mathcal{F}$ is differentiable at $r=0$.

The next proposition is essentially the same as proposition 9 in \cite{tr2}.  The key difference is in the definition of $a_{k,i}$ and $b_{k,i}$.  This difference plays out in the rest of the calculations of this section, which lead to the proof of the proposition.
Recall also that the $p_{k,i}$ form a periodic set of points with $p_{k+N,i}= p_{k,i} e^T$. This introduces a similar, but more obfuscated periodicity of the $a_{k,i}$ and $b_{k,i}$ below.

\begin{proposition}
Let $\{p_{k,i}\}$ be a balanced configuration.  Define $X_o$ by:
\[\alpha_{k,i}=\gamma_{k,i}=\beta_{k+1,i}=1/n_k,\]
\[a_{k,i}=\left(\conj^{k}(p_{k,i})\right)^{(-1)^k},\]
\[b_{k,i}=\left(\conj^k(p_{k-1,i})\right)^{(-1)^k}.\]
Then $\mathcal{F}(0,X_o)=0$.  Also, if $X$ is a solution to $\mathcal{F}(0,X)=0$ then, up to some identifications, $X=X_o$ for some balanced configuration $\{p_{k,i}\}$.  In addition, if $\{p_{k,i}\}$ is a non-degenerate balanced configuration then, up to some identifications, $D_2\mathcal{F}(0,X_o)$ is an isomorphism.  By the implicit function theorem, for t in a neighborhood of 0, there exists a unique $X(t)$ in a neighborhood of $X_o$ such that $\mathcal{F}(t,X(t))=0$. 
\label{main prop}
\end{proposition}

The Weierstrass data given by each unique $X(t)$ is the map of an immersed doubly periodic minimal surface with embedded planar ends.  The rest of this section contains the proof of Proposition \ref{main prop}.

\subsection{Solving the equation $\mathcal{F}_1=0$}
Assume $r=0$.  $\mathcal{F}_{1,k}=0$ is equivalent to: $G_kdz$ and $\eta_k$ have the same zeroes on ${\bar C}_k$.  Since they already have the same poles they are proportional.  By normalization, $\eta_k=\frac{1}{\delta_k z}G_k dz$.  Thus, $\mathcal{F}_1=0$ is equivalent to $\alpha_{k,i}=\gamma_{k,i}$ and $\beta_{k,i}=\gamma_{k-1,i}$.

From this moment on, assume that $\mathcal{F}_1=0$ so that $r=0\Rightarrow\eta_k=\frac{1}{\delta_k z}G_kdz$.
\subsection{Solving the equation $\mathcal{F}_2=0$}
Using proposition \ref{prop expansion1},
\[
\begin{split}
\mathcal{F}_{2,k,i} &=\frac{1}{\text{log}(r)}\text{Re}\int_{B_{k,i}}\eta \\
&=\frac{1}{\text{log}(r)}\text{Re}\left(\left(\gamma_{k,i}-\gamma_{k,1}\right)\text{log}(r)+\text{hol}(r,X)\right) \\
&=\gamma_{k,i}-\gamma_{k,1}+\frac{\text{Re}(\text{hol}(r,X))}{\text{log}(r)} 
\end{split}
\]
When $r=0$, $\mathcal{F}_{2,k,i}=\gamma_{k,i}-\gamma_{k,1}$.  Thus, $\mathcal{F}_2=0 \Rightarrow \gamma_{k,i}=\gamma_{k,1} \forall i \Rightarrow \gamma_{k,i}=\frac{1}{n_k} \forall i$.

\subsection{Solving the equation $\mathcal{F}_3=0$}
Using proposition \ref{prop expansion1},
\[
\begin{split}
\mathcal{F}_{3,k,i} =& \sqrt{r}\overline{\left(\overline{\int_{B_{k,i}}G^{-1}\eta}-\int_{B_{k,i}}G\eta\right)} \\
=& (-1)^k\text{\conj}^{k}\left(\int_{a_{k,1}}^{a_{k,1}}G_k^{-1}\eta_k+r\text{ log}(r)\text{hol}(r,X)+r\text{ hol}(r,X)\right) \\
&+(-1)^k\text{\conj}^{k+1}\left(\int_{b_{k+1,1}}^{b_{k+1,i}}G_{k+1}^{-1}\eta_{k+1}+r\text{ log}(r)\text{hol}(r,X)+r\text{ hol}(r,X)\right) 
\end{split}
\]
When r=0,
\[
\begin{split}
\mathcal{F}_{3,k,i} &=(-1)^k\text{\conj}^{k}\left(\int_{a_{k,1}}^{a_{k,i}}G_k^{-1}\eta_k\right)+(-1)^k\text{\conj}^{k+1}\left(\int_{b_{k+1}}^{b_{k+1,i}}G_{k+1}^{-1}\eta_{k+1}\right) \\
&=(-1)^k\text{\conj}^{k}\left(\int_{a_{k,1}}^{a_{k,i}}\frac{1}{\delta_k z}dz\right)+(-1)^k\text{\conj}^{k+1}\left(\int_{b_{k+1,1}}^{b_{k+1,i}}\frac{1}{\delta_{k+1}z}dz\right) \\
&=(-1)^k\text{\conj}^{k}\left(\delta_k^{-1}\log\left(\frac{a_{k,i}}{a_{k,1}}\right)\right)+(-1)^k\text{\conj}^{k+1}\left(\delta_{k+1}^{-1}\log\left(\frac{b_{k+1,i}}{b_{k+1,1}}\right)\right)
\end{split}
\]
Thus, $\mathcal{F}_3=0$ and $\delta_k=1$ for $k=1,2,\ldots,N \Rightarrow \frac{b_{k+1,i}}{b_{k+1,1}}=\text{\conj}\left(\frac{a_{k,1}}{a_{k,i}}\right)$.

\subsection{Solving the equation $\mathcal{F}_4=0$}
Using proposition \ref{prop expansion1},
\[
\begin{split}
\mathcal{F}_{4,k,i} =&\frac{(-1)^k}{\sqrt{r}}\left(\overline{\int_{A_{k,i}}G^{-1}\eta}-\int_{A_{k,i}}G\eta\right)\\\\
=&\frac{1}{\sqrt{r}}\left[\text{\conj}^{k+1}\int_{A_{k,i}}G^{(-1)^{k+1}}\eta-\text{\conj}^k\int_{A_{k,i}}G^{(-1)^k}\eta\right]\\\\
=&\frac{1}{\sqrt{r}}\text{\conj}^{k+1}\left[\sqrt{r}\left(-2\pi i \text{res}_{b_{k+1,i}}G_{k+1}\eta_{k+1}+r\text{ holo}(r,X)\right)\right]\\
&-\frac{1}{\sqrt{r}}\text{\conj}^k\left[\sqrt{r}\left(2\pi i\text{ res}_{a_{k,i}}G_k\eta_k+r\text{ holo}(r,X)\right)\right]\\\\
=&\text{\conj}^{k+1}\left[-2\pi i\text{ res}_{b_{k+1,i}}G_{k+1}\eta_{k+1}+r\text{ holo}(r,X)\right]\\
&-\text{\conj}^k\left[2\pi i\text{ res}_{a_{k,i}}G_k\eta_k+r\text{ holo}(r,X)\right].\\
\end{split}
\]
Thus, when $r=0$,
\[
\begin{split}
\mathcal{F}_{4,k,i} =&\text{\conj}^{k+1}\left[-2\pi i\text{ res}_{b_{k+1,i}}G_{k+1}\eta_{k+1}\right]-\text{\conj}^k\left[2\pi i\text{ res}_{a_{k,i}}G_k\eta_k\right]\\
=&2\pi i(-1)^{k}\left[-\text{ \conj}^{k}\left(\text{res}_{a_{k,i}}G_{k}\eta_{k}\right)+\text{ \conj}^{k+1}\left(\text{res}_{b_{k+1,i}}G_{k+1}\eta_{k+1}\right)\right]\\
=&-4\pi \delta_k i(-1)^k\text{\conj}^k\left(\sum_{j=1,\neq i}^{n_k}\frac{a_{k,i}}{n_k^2\left(a_{k,i}-a_{k,j}\right)}-\sum_{j=1}^{n_{k-1}}\frac{a_{k,i}}{n_k n_{k-1}\left(a_{k,i}-b_{k,j}\right)}\right) \\
&+4\pi\delta_{k+1} i(-1)^k\text{\conj}^{k+1}\left(-\sum_{j=1}^{n_{k+1}}\frac{b_{k+1,i}}{n_k n_{k+1}\left(b_{k+1,i}-a_{k+1,j}\right)}+\sum_{j=1,\neq i}^{n_k}\frac{b_{k+1,i}}{n_k^2\left(b_{k+1,i}-b_{k+1,j}\right)}\right). \\
\end{split}
\]
We will deal with this equation further in section 7.6 below.

\subsection{Solving the Equation $\mathcal{F}_5=0$}
Using proposition \ref{prop expansion2},
{\footnotesize
\[
\begin{split}
\overline{\int_{\partial D_{\epsilon}(0_k)}\frac{1}{G}\eta}-\int_{\partial D_{\epsilon}(0_k)}G\eta =&\left((-1)^k\text{\conj}^{k+1}\int_{\partial D_{\epsilon}(0_k)}G^{(-1)^k}\eta+(-1)^{k+1}\text{\conj}^k\int_{\partial D_{\epsilon}(0_k)}G^{(-1)^{k+1}}\eta\right)\\
=&\frac{1}{\sqrt{r}}(-1)^{k+1}\text{\conj}^k\left(2\pi i \text{ res}_{0}\frac{1}{G_k}\eta_{k}+\text{$r$ hol}(r,X)\right). \\
\end{split}
\]}
Thus,
\[
\begin{split}
\mathcal{F}_{5,k} =& 2\pi i \text{ res}_{0}\frac{1}{G_k}\eta_{k}+\text{$r$ hol}(r,X)\\
=&2\pi\delta_k^{-1} i+\text{$r$ hol}(r,X).\\
\end{split}
\]
When $r=0$, 
\[
\mathcal{F}_{5,k}=2\pi\delta_k^{-1} i,\;k=1,...,N.
\]
We want this period to be $2 \pi i$.  Thus, we need $\delta_k=1$ for $k=1,2,\ldots,N$.

\subsection{Uncovering the force equations and the non-horizontal period $\mathcal{T}_t$}

Our force equations could just be given by $\mathcal{F}_{4,k,i}$ for $k=1,\ldots,N$ and $i=1,\ldots,n_k$.  However, the non-horizontal period $\mathcal{T}_t$ whose limit is $\mathcal{T}_0=\overline{T}$ doesn't have a clear relationship to the points $(a,b)$.  Therefore, as done in \cite{tr2}, we will construct an isomorphism $(a,b)\mapsto(T,p,q)$.   

Let $m=n_1+\cdots+n_N$.  Given $p_{k,i}\in \C$, $k=1,\cdots,N$, $i=1,\cdots,n_k$, let $p\in\C^m$ be the vector whose components are $p_{k,i}$.  Given $(T,p,q)\in \C \times \C^m \times \C^m$, define $(a,b)$ by
\[
a_{k,i}=\left(\text{\conj}^{k}p_{k,i}q_{k,1}\right)^{(-1)^k} 
\]
\[
b_{k,i}=\left(\text{\conj}^{k}p_{k-1,i}q_{k,i}\right)^{(-1)^k}
\]
where $p_{k+N,i}=p_{k,i}e^T$ and $q_{k+N,i}=q_{k,i}$.\\

Note that the way $(a,b)$ are defined is similar to how they were defined in proposition \ref{main prop}.  We get the $(a,b)$ in proposition \ref{main prop} if we let $q_{k,i}=1$ for $k=1,\ldots,N$ and $i=1,\ldots,n_k$.  Also, our $(a,b)$ is a multiplicative version of the $(a,b)$ in \cite{tr2}.

If $\delta_k=1$ for $k=1,\ldots,N$ then 
\[
\begin{split}
\mathcal{F}_{3,k,i}&=(-1)^k\text{\conj}^{k}\left(\log\left(\frac{a_{k,i}}{a_{k,1}}\right)\right)+(-1)^k\text{\conj}^{k+1}\left(\log\left(\frac{b_{k+1,i}}{b_{k_1,1}}\right)\right)\\
&=\log{\left(\frac{p_{k,i}}{p_{k,1}}\right)}-\log{\left(\frac{p_{k,i}q_{k+1,i}}{p_{k,1}q_{k+1,1}}\right)}\\
&=\log{q_{k+1,i}}-\log{q_{k+1,1}}.
\end{split}
\]  

If $\mathcal{F}_{3,k,i}=0$ then  $\log{q_{k+1,i}}=\log{q_{k+1,1}}$.  Hence, $q_{k,i}=q_{k,1}$ for $k=1,\cdots,N$ and $i=1,\cdots,n_k$.  Thus, let $q_k=q_{k,1}$.\\ 

We finally deal with $\mathcal{F}_{4,k,i}$.  Assume $\mathcal{F}_2=0,\mathcal{F}_3=0$, and $\delta_k=1, k=1,2,\ldots N$.  Then, 
{\footnotesize
\[
\begin{split}
\frac{\mathcal{F}_{4,k,i}}{-4\pi i}=&(-1)^k\text{\conj}^k\left(\sum_{j=1,\neq i}^{n_k}\frac{a_{k,i}}{n_k^2\left(a_{k,i}-a_{k,j}\right)}-\sum_{j=1}^{n_{k-1}}\frac{a_{k,i}}{n_k n_{k-1}\left(a_{k,i}-b_{k,j}\right)}\right) \\
&-(-1)^k\text{\conj}^{k+1}\left(-\sum_{j=1}^{n_{k+1}}\frac{b_{k+1,i}}{n_k n_{k+1}\left(b_{k+1,i}-a_{k+1,j}\right)}+\sum_{j=1,\neq i}^{n_k}\frac{b_{k+1,i}}{n_k^2\left(b_{k+1,i}-b_{k+1,j}\right)}\right) \\
=&(-1)^k\sum_{j \neq i}^{n_k}\frac{(p_{k,i}q_k)^{(-1)^k}}{n_k^2\left((p_{k,i}q_k)^{(-1)^k}-(p_{k,j}q_k)^{(-1)^k}\right)}\\
&-(-1)^k\sum_{j=1}^{n_{k-1}}\frac{(p_{k,i}q_k)^{(-1)^k}}{n_k n_{k-1}\left((p_{k,i}q_k)^{(-1)^k}-(p_{k-1,j}q_k)^{(-1)^k}\right)} \\
&+(-1)^{k}\sum_{j=1}^{n_{k+1}}\frac{(p_{k,i}q_{k+1})^{(-1)^{k+1}}}{n_k n_{k+1}\left((p_{k,i}q_{k+1})^{(-1)^{k+1}}-(p_{k+1,j}q_{k+1})^{(-1)^{k+1}}\right)}\\
&+(-1)^{k+1}\sum_{j \neq i}^{n_k}\frac{(p_{k,i}q_{k+1})^{(-1)^{k+1}}}{n_{k}^2\left((p_{k,i}q_{k+1})^{(-1)^{k+1}}-(p_{k,j}q_{k+1})^{(-1)^{k+1}}\right)} \\
=&\sum_{j \neq i}\frac{p_{k,i}+p_{k,j}}{n_k^2(p_{k,i}-p_{k,j})}+(-1)^k\left(\sum_{j=1}^{n_{k+1}}\frac{p_{k+1,j}^{(-1)^k}}{n_kn_{k+1}\left(p_{k+1,j}^{(-1)^k}-p_{k,i}^{(-1)^k}\right)}-\sum_{j=1}^{n_{k-1}}\frac{p_{k,i}^{(-1)^k}}{n_kn_{k-1}\left(p_{k,i}^{(-1)^k}-p_{k-1,j}^{(-1)^k}\right)}\right).
\end{split}
\]}

Thus, assuming $\mathcal{F}_2=0$ and $\mathcal{F}_3=0$, we get $\mathcal{F}_{4,k,i}=-4\pi i(-1)^k F_{k,i}$.  Now, if $\{p_{k,i}\}$ is a balanced configuration then define $X_o$ as in the statement of Proposition \ref{main prop}.  Because of $q_{k,i}=1$, we get $\mathcal{F}(0,X_o)=0$, proving the first statement of Proposition \ref{main prop}.

In order to prove the converse, it is necessary to make some identifications since $\mathcal{F}_3=0$ only implies that $q_{k,i}=q_{k,1}$.  We need $q_{k,1}=1$.  Our identifications are multiplicative versions of the similar identifications in section $6.5$ of \cite{tr2}.  Given complex numbers $\lambda_k$, let $a_{k,i}'=a_{k,i}\lambda_k$ and $b_{k,i}'=b_{k,i}\lambda_k$.  Then $\mathcal{ F}_3'=\mathcal{F}_3$ and $\mathcal{F}_4'=\mathcal{F}_4$.  Let $\left(\Sigma',G',\eta'\right)$ be the Weierstrass data corresponding to $a_{k,i}'$ and $b_{k,i}'$.  Then the map $\phi:\Sigma \rightarrow \Sigma',z \in \overline{\C}_k \mapsto z \lambda_k$ is an isomorphism with $\phi^*G'dz=Gdz$ and $\phi^*\eta'=\eta$.  Thus, the Weierstrass data $\left(\Sigma,G,\eta\right)$ and $\left(\Sigma',G',\eta'\right)$ are isomorphic and define equivalent minimal surfaces.  Hence, the above identification makes sense:
\[(a,b) \sim (a',b') \Longleftrightarrow \forall k \; \exists \lambda_k \text{ such that } \forall i,\; a_{k,i}'=a_{k,i}\lambda_k,\; b_{k,i}'=b_{k,i}\lambda_k.\]
We can create similar identifications for $p$ and $q$:
\[p' \sim p \Longleftrightarrow \exists \lambda \text{ such that } \forall k,i, \; p_{k,i}'=p_{k,i}\lambda\]
\[ q' \sim q \Longleftrightarrow \forall k \; \exists \lambda_k \text{ such that } \forall i,\; q_{k,i}'=q_{k,i}\lambda_k.\]

As simple computations yields 
\begin{lemma}
The map $(T,p,q) \mapsto (a,b)$ is an isomorphism. $\Box$
\end{lemma}

Using the identifications on $(a,b)$, $p$, and $q$, we get that $\mathcal{F}_3=0 \Rightarrow q_{k,1}  \sim 1$.  This  proves the second part of Proposition \ref{main prop}.

\subsection{$D_2\mathcal{F}(0,X_0)$ is an isomorphism}
The next three lemma's are from \cite{tr2}.  Lemmas \ref{lemma2} and \ref{lemma3} are the same as propositions $10$ and $11$ in \cite{tr2}.  Our lemma \ref{lemma4} is partly proven in section 6.5 of \cite{tr2}.  
\begin{lemma}[\cite{tr2}]
Let $E=\{(\alpha_k',\beta_k')\in\mathbb{C}^{n_k+n_{k-1}} | \sum\alpha_{k,i}'=\sum\beta_{k,i}'=0\}$.  The partial differential of $\mathcal{F}_{1,k}$ with respect to $(\alpha_k,\beta_k)$ is an isomorphism from $E$ onto $\mathbb{C}^{n_k+n_{k-1}-2}$.
\label{lemma2}
\end{lemma}
\begin{proof}
See proposition $10$ in section $6.2$ of \cite{tr2}. 
\end{proof}

\begin{lemma}[\cite{tr2}]
$$\sum_{k=1}^N\sum_{i=1}^{n_k}\mathcal{F}_{4,k,i}(t,X)=0\;\forall(t,X).$$
\label{lemma3}
\end{lemma}
\begin{proof}
See proposition $11$ in section $6.5$ of \cite{tr2}.
\end{proof}

\begin{lemma}[\cite{tr2}]
The partial differential of $\mathcal{F}$ evaluated at $(0,X_0)$ with respect to the variables $(\alpha,\beta),\gamma,q,p,\delta$ has the form
\[
\begin{bmatrix}\mathcal{ I}_1 & \cdot & 0 & 0 & 0  \\ 0 & \mathcal{I}_2 & 0 & 0 & 0 \\ \cdot & \cdot &\mathcal{I}_3 & 0 & \cdot \\ \cdot & \cdot & \cdot &\mathcal{I}_4 & \cdot \\ \cdot & \cdot & 0 & 0 & \mathcal{I}_5 \end{bmatrix}
\]
with $\mathcal{I}_k$ an invertible linear operator for $k=1,2,3,4,5$, and so it is invertible.
\label{lemma4}
\end{lemma}
\begin{proof}
The arguments explaining the first four entries of the top four rows are explained in section $6.5$ of \cite{tr2}.  We repeat those arguments.  The key difference is that there is no fifth row or column in \cite{tr2}.  

In the first row, $\mathcal{I}_1$ is invertible by lemma \ref{lemma2}.  If $\alpha_k=\gamma_k$ and $\beta_k=\gamma_{k-1}$ then $\eta_k=\frac{1}{\delta_k z}G_kdz$, and so $\mathcal{F}_1=0$ independent of $q,p$, and $\delta$.  Hence, there are zeroes in the last three entries of the first row.

The second row is clear because $\mathcal{F}_{2,k,i}=\gamma_{k,i}-\gamma_{k,1}$ when $r=0$ and is independent of $\alpha ,\beta ,q,p$, and $\delta$.

The identification on $q$ makes $\mathcal{I}_3$ invertible.  The zero in the third row is because $\mathcal{F}_3$ is independent of $p$.

By lemma \ref{lemma3}, we can think of $\mathcal{ F}_4$ as a map into the subspace $\sum\mathcal{F}_{4,k,i}=0$.  Also, $$\mathcal{I}_4=4\pi i(-1)^{k+1}\frac{\partial}{\partial p}F.$$
Thus, the non-degeneracy of the force equations implies that $\mathcal{I}_4$ is onto.  The identification on $p$ implies that $\mathcal{I}_4$ is invertible.

When $r=0,\alpha_k=\gamma_k$, and $\beta_k=\gamma_{k-1}$, we get $\mathcal{F}_{5,k}=\frac{2\pi i}{\delta_k}$.  Thus, $\mathcal{I}_5$ is invertible.  The zeroes in row five are due to the fact that $\mathcal{F}_5$ is independent of $p$ and $q$ when $r=0,\alpha_k=\gamma_k$, and $\beta_k=\gamma_{k-1}$.
\end{proof}

Finally, we have shown that $D_2 F(0,X_0)$ is an isomorphism, completing the proof of proposition \ref{main prop}.  There are the two free parameters $t\in\mathbb{R}$ and $T\in\mathbb{C}$.  Thus, the implicit function theorems provides a three-dimensional space of solutions to the equation $\mathcal{F}(t,X)=0$.  As discussed in \cite{hatr1}, this is the expected size of our space of minimal surfaces. Note that in  \cite{tr2}, the surfaces are made up of domains $\C_k$. 
The balance configurations can be changed by complex linear transformation that do not affect the minimal surface. In our case, the domains are 
 punctured planes $\C_k^*$, and the balance configurations can only be changed by complex multiplications. This explains the difference in the dimensions of the moduli spaces.

\section{Embeddedness and properties \ref{property1} and \ref{property2}}
We can use the technique in \cite{tr2} to prove that our surfaces are embedded.  The only variation is that our surfaces have pairs of ends at each level.  However, it turns out this is a minor difference when it comes to proving embeddedness.  In the process of proving embeddedness, we also show that the surfaces satisfy properties \ref{property1} and \ref{property2}.

Let $\left(\Sigma,G,\eta\right)$ be the Weierstrass data given by proposition \ref{main prop} for some small positive $t$.  In this section, it is convenient to express $\psi$ as
\[
\psi(z)=\left(\text{horiz}(z),\text{height}(z)\right)\in \mathbb{C}\times\mathbb{R}.
\]

The following proposition is essentially the same as proposition $12$ in section $7$ of \cite{tr2}.  Parts $4,5$, and $6$ have slight differences.  We include a calculation of the location of the ends at each level.  
 
\begin{proposition}
There exists a constant $C$, not depending on $t$, such that:
\begin{enumerate}
\item
For any point $z\in {\bar \C}_k$ such that $\forall i$, $\abs{v_{k,i}}>\epsilon$, $\abs{w_{k,i}}>\epsilon$,
\[
\abs{\text{height}(z)-\text{height}(\infty_k)}\leq C.
\]
\item
For any point $z\in {\bar C}_k$ such that $\frac{r}{\epsilon}<\abs{v_{k,i}}<\epsilon$,
\[
\abs{\text{height}(z)-\text{height}(\infty_k)-\frac{1}{n_k}\log{\abs{v_{k,i}(z)}}}\leq C.
\]
\item
\[
\abs{\text{height}(\infty_{k+1})-\text{height}(\infty_k)-\frac{1}{n_k}\log{r}}\leq C.
\]
\item
Choose $P_{k,i}\in\Sigma$ such that $v_{k,i}(P_{k,i})=\sqrt{r}$.  Note that $G(P_{k,i})=1$.  Then
\[
2\sqrt{r}\left(\text{horiz}(P_{k,j})-\text{horiz}(P_{k,i})\right)\rightarrow(-1)^k\text{ \conj}^{k+1}(a_{k,j}-a_{k,i})=\log{\overline{p_{k,j}}}-\log{\overline{p_{k,i}}}
\]
and
\[
2\sqrt{r}\left(\text{horiz}(P_{k,j})-\text{horiz}(P_{k-1,i})\right)\rightarrow(-1)^k\text{ \conj}^{k+1}(a_{k,j}-b_{k,i})=\log{\overline{p_{k,j}}}-\log{\overline{p_{k-1,i}}}.
\]
Thus, we can translate the surface such that $2\sqrt{r}\text{ horiz}(P_{k,i})\rightarrow \log{\overline{p_{k,i}}}\;\forall k,i$.
\item
Let $0<\sigma<\frac{1}{2}$.  The image of the domain $r^{1-\sigma}<\abs{v_{k,i}}<r^\sigma$ converges to a catenoid with necksize $\frac{2\pi}{n_k}$, and it is contained in a vertical cylinder with radius $\frac{5r^{\sigma-1/2}}{n_k}$.
\item
The non-horizontal period of $\psi$ is
\[
\mathcal{T}=\text{Re}\int_{B_{1,1}}\phi\simeq \left(\frac{\overline{T}}{2\sqrt{r}},\left(\sum_{k=1}^N\frac{1}{n_k}\right)\log{r}\right).
\]
\item
For each $k=1,\ldots,N$,
\[
2\sqrt{r}\text{Re}(\text{horiz}(0_k))\rightarrow (-1)^{k+1}\infty
\]
and
\[
2\sqrt{r}\text{Re}(\text{horiz}(\infty_k))\rightarrow (-1)^k\infty.
\]
\end{enumerate}
\label{embed prop1}
\end{proposition}

\begin{proof}
The proof of this proposition uses the same techniques used in the proof of proposition $8$ in section $5$ of \cite{tr2}.  We show the details of the proof of part $7$.  

Let $z_k\in{\bar \C}_k$ be the point such that $v_{k,1}(z_k)=\epsilon$ for $k=1,\cdots,N$, and let $z_1$ be the base point for $\psi$.  Suppose $z\in\overline{\mathbb{C}}_k$.  Since $\text{Re}(\text{horiz}(z))=\text{Re}\int_{z_1}^{z_k}\phi_1+\text{Re}\int_{z_k}^z\phi_1$ and $\text{Re}\int_{z_1}^{z_k}\phi_1$ is bounded, we only need to consider $\text{Re}\int_{z_k}^z\phi_1$.  In that case,
\[
\begin{split}
2\sqrt{r}\text{Re}(\text{horiz}(z)) &=2\sqrt{r}\text{Re}\int_{z_k}^z\frac{1}{2}\left(\frac{1}{G}-G\right)\eta\\
&=\text{Re}\sqrt{r}(-1)^k\int_{z_k}^z\left(\frac{1}{\sqrt{r}G_k}-\sqrt{r}G_k\right)\eta\\
&=\text{Re}\sqrt{r}(-1)^k\int_{z_k}^z\frac{1}{\sqrt{r}G_k}\left(\eta_k+r\text{ holo}(r,X)dz\right)+\text{Re}\sqrt{r}(-1)^{k+1}\int_{z_k}^z\sqrt{r}G_k\eta\\
&=\text{Re}\left((-1)^k\int_{z_k}^z\frac{1}{\delta_k z}dz+r\text{ holo}(r,X)\right)\\
&=(-1)^k\left(\log{\abs{z}}-\log{\abs{z_k}}\right)+\text{Re}\left(r\text{ holo}(r,X)\right)\\
&\rightarrow(-1)^k\left(\log{\abs{z}}-\log{\abs{z_k}}\right).
\end{split}
\] 
Thus,
\[
2\sqrt{r}\text{Re}(\text{horiz}(\infty_k))\rightarrow(-1)^k\infty
\]
and
\[
2\sqrt{r}\text{Re}(\text{horiz}(0_k))\rightarrow(-1)^{k+1}\infty.
\]
\end{proof}

\begin{proposition}[\cite{tr2}]
For small $t>0$, the minimal surface given by proposition $\ref{main prop}$ is embedded.
\label{embed prop2}
\end{proposition}

The same proposition is proven in section $7$ of \cite{tr2}.  The only difference in the proof is due to the fact that we have two ends at each level instead of one.  In \cite{tr2}, Traizet splits $\mathbb{R}^3$ into the horizontal slabs
\[
\text{height}(\infty_{k+1})+\frac{\sigma}{n_{k+1}}\left|\log{r}\right|\leq x_3 \leq \text{height}(\infty_k)-\frac{\sigma}{n_k}\left|\log{r}\right|.
\]
and
\[
\text{height}(\infty_k)-\frac{\sigma}{n_k}\abs{\log{r}}\leq x_3 \leq \text{height}(\infty_k)+\frac{\sigma}{n_k}\abs{\log{r}}.
\]
Traizet shows that the intersection of the first slab with $\psi(\Sigma)$ is the $n_k$ disjoint components $C_{k,i,t}$, each one converging to a catenoid.  Therefore, this portion of the surface is embedded.  

Then, he shows that the intersection of the second slab with $\psi(\Sigma)$ is the region $E_{k,t}$, which is a graph over the plane and hence embedded.  The difference here is that we have two embedded ends in $E_{k,it}$.  However, by part $7$ of proposition \ref{embed prop1}, $\text{horiz}(0_k)=(-1)^{k+1}\infty$ and $\text{horiz}(\infty_k)=(-1)^k\infty$.  Hence, the ends in each level are disjoint.  Thus, we get that $E_{k,t}$ is embedded.  Therefore, $\psi(\Sigma)$ is embedded.  Proposition \ref{embed prop1} together with the proof of proposition \ref{embed prop2} show that our surfaces satisfy properties \ref{property1} and \ref{property2}.

\addcontentsline{toc}{section}{References}
\bibliographystyle{plain}
\bibliography{minlit}

\label{sec:liter}

\end{document}